\documentclass[12pt,a4paper]{article}
\usepackage{amsmath}
\usepackage{amssymb}
\usepackage{amssymb}
\usepackage{cases}
\usepackage{bbm}
\usepackage{mathrsfs}
\usepackage{graphicx}
\usepackage{amsfonts}
\usepackage{enumerate}
\usepackage{theorem}
\usepackage[pdfstartview=FitH]{hyperref}
\makeatletter
\def\tank#1{\protected@xdef\@thanks{\@thanks
 \protect\footnotetext[0]{#1}}}
\def\bigfoot{

 \@footnotetext}
\makeatother

\newcommand{\ea}{\end{array}}

\allowdisplaybreaks

\newtheorem{theorem}{Theorem}[section]
\newtheorem{lem}{Lemma}[section]
\newtheorem{prp}[theorem]{Proposition}
\newtheorem{thm}[theorem]{Theorem}
\newtheorem{cor}[theorem]{Corollary}
\newtheorem{dfn}[theorem]{Definition}

\def\beq{\begin{equation}}
\def\nneq{\end{equation}}

\def\bthm{\begin{thm}}
\def\nthm{\end{thm}}

\def\blem{\begin{lem}}
\def\nlem{\end{lem}}
\def\bprf{\begin{proof}}
\def\nprf{\end{proof}}
\def\bprop{\begin{prop}}
\def\nprop{\end{prop}}
\def\brmk{\begin{rem}}
\def\nrmk{\end{rem}}

\def\bexa{\begin{exa}}
\def\nexa{\end{exa}}
\def\bcor{\begin{cor}}
\def\ncor{\end{cor}}

\def\AA{\mathcal A}

\def\FF{\mathcal F}
\def\EE{\mathcal E}

\def\HH{\mathcal H}

\def\e{\varepsilon}

{\theorembodyfont{\rmfamily}
}

\title{Moderate deviations for stochastic models of two-dimensional second grade fluids}

\footnotesize{\author{Jianliang Zhai $^{1,}$\thanks{zhaijl@ustc.edu.cn},\ \
 Tusheng Zhang$^{2,1,}$\thanks{Tusheng.Zhang@manchester.ac.uk},\ \
 Wuting Zheng$^{1,}$\thanks{zwtzjr@mail.ustc.edu.cn}\\
 {\em $^1$ School of Mathematical Sciences,}\\
 {\em University of Science and Technology of China,}\\
 {\em Hefei, 230026, China}\\
 {\em $^2$ School of Mathematics, University of Manchester,}\\
 {\em Oxford Road, Manchester, M13 9PL, UK}\\
}
\date{}
\newenvironment{proof}{\par\noindent{\bf Proof:}}{\hspace*{\fill}$\blacksquare$\par}
\begin{document}
\maketitle
\noindent \textbf{Abstract:}
In this paper, we prove a central limit theorem and establish a moderate deviation principle for stochastic models of incompressible second grade fluids.
The weak convergence method introduced by \cite{Budhiraja-Dupuis} plays an important role.


\vspace{4mm}


\vspace{3mm}
\noindent \textbf{Key Words:}
Central limit theorem;
Moderate deviations;
Second grade fluids;
Non-Newtonian fluid;
Stochastic partial differential equations.
\numberwithin{equation}{section}
\section{Introduction}

The second grade fluids is an admissible
model of slow flow fluids, which contains a large class  of Non-Newtonian fluids such as industrial fluids, slurries, polymer melts, etc..
And as mentioned in \cite{Dunn-Fosdick}, ``the second grade fluid has general and pleasant properties such as boundedness, stability, and exponential decay". Furthermore, it also has interesting connections with many other fluid models, see \cite{Busuioc}, \cite{Busuioc-Ratiu}, \cite{Holm-Marsden-Ratiu}, \cite{Holm-Marsden-Ratiu01}, \cite{Iftimie}, \cite{Shkoller2001}, \cite{Shkoller1998} and references therein. For example, the second grade fluids reduce to Navier-Stokes Equations when some of the parameters equal to 0, and it was shown in \cite{Iftimie} that they are good approximations of the Navier-Stokes Equation.
We refer to \cite{Dunn-Fosdick}, \cite{Dunn-Rajagopal}, \cite{Fosdick-Rajagopal}, \cite{Noll-truesdell} for a comprehensive theory of the second grade fluids.

Recently, the stochastic models of two-dimensional second grade fluids (\ref{01}) have been studied in \cite{RS-12}, \cite{RS-10-01}  and \cite{RS-10}, where the authors obtained the existence and uniqueness of solutions
and investigated the behavior of the solutions. The martingale solution of
such system driven by L\'evy noise is studied in \cite{HRS}.

\vskip 0.3cm
In this paper, we are concerned with asymptotic behaviors of stochastic models for the incompressible second grade
fluids, which are given as follows:
\begin{eqnarray}\label{01}
&  &d(u^\varepsilon-\alpha \Delta u^\varepsilon)+\Big(-\nu \Delta u^\varepsilon+curl(u^\varepsilon-\alpha \Delta u^\varepsilon)\times u^\varepsilon+\nabla\mathfrak{P}^\e\Big)dt\\\nonumber
& &=
F(u^\varepsilon,t)dt+\sqrt{\varepsilon}G(u^\varepsilon,t)dW,\ \ \ in\ \mathcal{O}\times(0,T],
\end{eqnarray}
under the following condition
\begin{eqnarray}\label{02}
\left\{
 \begin{array}{llll}
 & \hbox{${\rm{div}}\ u^\varepsilon=0\ \text{in}\ \mathcal{O}\times(0,T]$;} \\
 & \hbox{$u^\varepsilon=0\ \text{in}\ \partial \mathcal{O}\times[0,T]$;} \\
 & \hbox{$u^\varepsilon(0)=u_0\ \text{in}\ \mathcal{O}$,}
 \end{array}
\right.
\end{eqnarray}
where $\mathcal{O}$ is a connected, bounded open
subset of $\mathbb{R}^2$ with boundary $\partial \mathcal{O}$ of class $\mathcal{C}^3$; $u^\e=(u_1^\e,u_2^\e)$ and $\mathfrak{P}^\e$ represent the random velocity and modified pressure, respectively;
and $W$ is an $m$-dimensional standard Brownian motion defined on a complete probability space $(\Omega,\mathcal{F},\{\mathcal{F}_t\}_{t\in[0,T]},P)$.

\vskip 0.2cm
Set
\begin{eqnarray}\label{SP-01}
\mathcal{C}=\Big\{u\in[\mathcal{C}^\infty_c(\mathcal{O})]^2\ {\rm such \ that\  div}\ u=0\Big\},\nonumber\\
\mathbb{H}={\rm\ closure\ of}\ \mathcal{C}\ {\rm in}\ \mathbb{L}^2(\mathcal{O}).\nonumber
\end{eqnarray}
Let $\mathbb{P}$ be the Helmholtz-Leray projection from $\mathbb{L}^2(\mathcal{O})$ into $\mathbb{H}$. Let $A$ be the Stoke operator $-\mathbb{P}\Delta$(see the precise definition below). One can see that (\ref{01}) is equivalent to the following stochastic evolution equation:
\begin{eqnarray}\label{Abstract}
du^\varepsilon(t)+\nu \widehat{A}u^\varepsilon(t)dt+\widehat{B}(u^\varepsilon(t),u^\varepsilon(t))dt=\widehat{F}(u^\varepsilon(t),t)+\sqrt{\varepsilon}\widehat{G}(u^\e(t),t)dW(t),
\end{eqnarray}
with initial value $u^{\e}(0)=u_0$.\\
Where $\widehat{A}=(I+\alpha A)^{-1}A$, $\widehat{B}(u,v)=(I+\alpha A)^{-1}\Big(curl(u-\alpha\Delta u)\times v\Big)$, $\widehat{F}=(I+\alpha A)^{-1}F$, $\widehat{G}=(I+\alpha A)^{-1}G$.

\vskip 0.3cm
As the parameter $\varepsilon$ tends to zero, the solution $u^\e$ of (\ref{Abstract}) will tend to the solution of the following deterministic equation
\begin{eqnarray}\label{Deterministic}
du^0(t)+\nu \widehat{A}u^0(t)dt+\widehat{B}(u^0(t),u^0(t))dt=\widehat{F}(u^0(t),t),
\end{eqnarray}
with initial value $u^0(0)=u_0$.

\vskip 0.2cm
In this paper, we shall investigate deviations of $u^{\e}$ from the deterministic solution $u^0$, as $\e$ decreases to $0$, that is, the asymptotic behavior of the trajectory,
$$Z^{\e}(t)=\frac{1}{\sqrt{\e}\lambda(\varepsilon)}(u^{\e}-u^0)(t),\ \ t\in[0,T],$$
where $\lambda(\e)$ is some deviation scale which strongly influences the asymptotic behavior of $Y^\e$.

\vskip 0.2cm
1. The case $\lambda(\e)=1/\sqrt{\e}$ provides some large deviations estimates. The large deviation theory has important applications in many areas,
such as in thermodynamics, statistical mechanics, information
theory and risk management, etc., see  \cite{Dembo-Zeitouni} \cite{Touchette} and reference therein, and it has been extensively studied in recent years. For stochastic evolution equations and stochastic
 partial differential equations driven by Gaussian processes, there are many papers on this topic, see e.g.\ \cite{BDM08}, \cite{BDM10}, \cite{BDF12}, \cite{CW}, \cite{CR}, \cite{CM2}, \cite{Liu}, \cite{S}, \cite{Z}.
Large deviations for stochastic models of two-dimensional second grade fluids has been obtained by Zhai and Zhang in \cite{Zhai}.

\vskip 0.2cm
2. The case $\lambda(\e)=1$ is known as the central limit theorem(CLT for short). We will show that $(u^{\e}-u^0)/\sqrt{\e}$ converges to a solution of a stochastic evolution equation as $\e$ decreases to $0$.

\vskip 0.2cm
3. When the deviation scale satisfies
\begin{equation}\label{condition}
\lambda(\e)\rightarrow+\infty,\ \ \ \ \sqrt{\e}\lambda(\e)\rightarrow 0 \ \ \ \ \text{as}\ \ \e\rightarrow 0,
\end{equation}
we are in the domain of the so-called moderate deviation principle (MDP for short, cf.\cite{Dembo-Zeitouni}), which
fills in the gap between the CLT scale $[\lambda(\e)=1]$ and the large deviations scale $[\lambda(\e)=1/\sqrt{\e}]$.
In this paper, we will establish the MDP for 
(\ref{01}).

Like the large deviations, the estimates of moderate deviations are very useful in the theory of statistical inference. It can provide us with the rate of convergence and a useful method for constructing asymptotic confidence intervals, see \cite{Ermakov}, \cite{Gao}, \cite{Inglot}, \cite{Kallenberg}  and references therein. There are many results on the MDP in various frameworks, for example, De Acosta \cite{Acosta}, Chen\cite{Chen} and Ledoux \cite{Ledoux} for processes with independent increments; Wu \cite{Wu} for Markov processes; Guillin and Liptser \cite{Guillin} for diffusion processed; Wang and Zhang \cite{Wang2} for stochastic reaction-diffusion equations; Wang, Zhai and Zhang \cite{Wang1} for 2-D stochastic Navier-Stokes equations driven by Brownian motion; Budhiraja, Dupuis and Ganguly \cite{Budhiraja-Dupuis-Ganguly} for stochastic differential equations with jump; Dong, Xiong, Zhai and Zhang \cite{DXZZ} for stochastic Navier-Stokes equations driven by Poisson random measures.


\vskip 0.2cm
To establish the MDP, we will adopt the weak convergence approach introduced in \cite{Budhiraja-Dupuis}, which has been used by many authors in the framework of non-linear hydrodynamics models driven by Gaussian noise, see for example \cite{Chueshov Millet}, \cite{Bessaih Millet}, \cite{Rockner Zhang} and \cite{Sritharan Sundar}. This approach amounts to establishing the weak convergence of perturbations of equation (\ref{01}) in the random directions of the Cameron-Martin space of the driving Brownian motion. Because of the nature of the second grade fluids models, to get the uniform bound(Lemma \ref{MDP lemma 02}) for the solutions of the random perturbations of the system, we are forced to work with the Galerkin approximations, rather than the equations themselves. The proof is long and quite technical. We are only
able to show that the solution family of the random perturbations of the system (\ref{01}) is tight in a larger space.
However, this turns out to be sufficient for us to prove the weak convergence in the actual state space with a stronger topology.

\vskip 0.3cm

 The organization of this paper is as follows. In Section 2, we introduce some functional spaces and state some estimates which will be used later. In
  Section 3, we formulate the hypotheses and recall the precise definition of solutions. We also collect some results on existence, uniqueness and regularities of the solutions. In Section 4, we establish the central limit theorem. Section 5 is devoted to establishing the moderate deviation principle.

  \vskip 0.2cm

Throughout this paper, $C,C_{p,T},C_{N}...$ denote positive constants depending on some parameters $p,T,N,...$, whose value may be different from line to line.

\section{Preliminaries}

In this section, we will introduce some functional spaces and preliminary facts that are needed in the paper.

For $p\geq 1$ and $k\in\mathbb{N}$, we denote by $L^p(\mathcal{O})$
and $W^{k,p}(\mathcal{O})$ the usual $L^p$ and Sobolev spaces over $\mathcal{O}$, and write $H^k(\mathcal{O}):=W^{k,2}(\mathcal{O})$.
Let $W^{k,p}_0(\mathcal{O})$ be
the closure in $W^{k,p}(\mathcal{O})$ of $\mathcal{C}^\infty_c(\mathcal{O})$ the space of infinitely differentiable functions with compact supports in $\mathcal{O}$, and denote $W^{k,2}_0(\mathcal{O})$ by $H_0^k(\mathcal{O})$. We equip $H^1_0(\mathcal{O})$
with the scalar product
\begin{eqnarray}\label{H-01}
((u,v))=\int_\mathcal{O}\nabla u\cdot\nabla vdx=\sum_{i=1}^2\int_\mathcal{O}\frac{\partial u}{\partial x_i}\frac{\partial v}{\partial x_i}dx,
\end{eqnarray}
where $\nabla$ is the gradient operator. It is well known that the norm $\|\cdot\|$ generated by this scalar product is equivalent to the usual norm of $W^{1,2}(\mathcal{O})$
in $H^1_0(\mathcal{O})$.

In the sequel, we denote the space $\{(x_1,x_2),\ x_1,x_2\in X\}$ by $\mathbb{X}$. Set
\begin{eqnarray}\label{SP-01}
\mathbb{V}={\rm\ closure\ of}\ \mathcal{C} {\rm\ in}\ \mathbb{H}^1(\mathcal{O}).
\end{eqnarray}
We denote by $(\cdot,\cdot)$ and $|\cdot|$ the inner product in $\mathbb{L}^2(\mathcal{O})$(in $\mathbb{H}$) and the induced norm, respectively. The inner product and the norm of $\mathbb{H}^1_0(\mathcal{O})$
are denoted respectively by $((\cdot,\cdot))$ and $\|\cdot\|$. For the space $\mathbb{V}$ we will use the norm generated
by the following scalar product
$$
(u,v)_\mathbb{V}=(u,v)+\alpha ((u,v)),\ \text{for any } u,v\in\mathbb{V}.
$$
The $\rm Poincar\acute{e}$'s inequality implies that, for some $\mathcal{P}>0$
\begin{eqnarray}\label{a}
(\mathcal{P}^2+\alpha)^{-1}\|v\|^2_\mathbb{V}
\leq
\|v\|^2
\leq
\alpha^{-1}\|v\|^2_\mathbb{V},\ \ for\ any\ v\in\mathbb{V}.
\end{eqnarray}

\vskip 0.2cm
We also introduce the following space
\begin{eqnarray*}
\mathbb{W}=\{u\in\mathbb{V}\ \text{such that }curl (u-\alpha\Delta u)\in L^2(\mathcal{O})\},
\end{eqnarray*}
and endow it with the norm generated by the scalar product
\begin{eqnarray}\label{W}
(u,v)_\mathbb{W}=(u,v)_\mathbb{V}+\Big(curl(u-\alpha\Delta u),curl(v-\alpha\Delta v)\Big).
\end{eqnarray}
 The following result can be found
in \cite{CO}, \cite{CG}. See also Lemma 2.1 in \cite{RS-12}.
\vskip 0.3cm

\begin{lem}
Set
$$
\widetilde{\mathbb{W}}=\Big\{v\in\mathbb{H}^3(\mathcal{O})\text{ such that }{\rm div} v=0\ and\ v|_{\partial \mathcal{O}}=0\Big\}.
$$
Then the following (algebraic and topological) identity holds:
\begin{eqnarray}\label{W = W}
\mathbb{W}=\widetilde{\mathbb{W}}.
\end{eqnarray}

Moreover, the following inequality holds: for some $C>0$
\begin{eqnarray}\label{W-02}
    \|v\|^2_{\mathbb{H}^3(\mathcal{O})}
\leq
    C\|v\|^2_\mathbb{W},\ \ \ \forall v\in \widetilde{\mathbb{W}}.
\end{eqnarray}

\end{lem}


If we identify the Hilbert space $\mathbb{V}$ with its dual space $\mathbb{V}^*$ by the Riesz representation, then we obtain a
Gelfand triple
\begin{eqnarray}\label{Gelfand}
\mathbb{W}\subset \mathbb{V}\subset\mathbb{W}^*.
\end{eqnarray}

 We denote by $\langle f,v\rangle$ the dual pair between $f\in\mathbb{W}^*$  and $v\in\mathbb{W}$. Then we have
$$(v,w)_\mathbb{V}=\langle v,w\rangle,\ \ \ \forall v\in\mathbb{V},\ \ \forall w\in\mathbb{W}.$$

Since the injection of $\mathbb{W}$ into $\mathbb{V}$ is compact,
there exists a sequence $\{(e_i,\lambda_i)\in \mathbb{W}\times \mathbb{R}:i\in\mathbb{N}\}$ which has the following properties
\begin{itemize}
\item[(1)] $\{e_i,\ i\in\mathbb{N}\}$ forms an orthonormal basis in $\mathbb{W}$,
and an orthogonal system in $\mathbb{V}$;

\item[(2)] $0<\lambda_i\uparrow\infty$;

\item[(3)] the elements of this sequence are the solutions of the eigenvalue problem
\begin{eqnarray}\label{Basis}
(v,e_i)_{\mathbb{W}}=\lambda_i(v,e_i)_{\mathbb{V}},\ \text{for any }v\in\mathbb{W}.
\end{eqnarray}
\end{itemize}

By Lemma 4.1 in \cite{RS-10}, we have
\begin{eqnarray}\label{eq Basis}
e_i\in\mathbb{H}^4(\mathcal{O}),\ \ \forall i\in\mathbb{N}.
\end{eqnarray}


 Consider the following ``generalized Stokes equations":

\begin{eqnarray}\label{General Stokes}
v-\alpha \Delta v=f\ {\rm in}\ \mathcal{O},\nonumber\\
{\rm div}\ v=0\ {\rm in}\ \mathcal{O},\\
v=0\ {\rm on}\ \partial \mathcal{O}.\nonumber
\end{eqnarray}

The following result can be derived from \cite{SO1}, \cite{SO2} and also can be found in \cite{RS-10} and \cite{RS-12}.
\begin{lem}\label{Lem GS}
Let $\mathcal{O}$ be a connected, bounded open subset of $\mathbb{R}^2$ with a boundary $\partial \mathcal{O}$ of class $\mathcal{C}^l$
and let $f$ be a function in $\mathbb{H}^l$, $l\geq 1$. Then the system (\ref{General Stokes}) has a solution $v\in \mathbb{H}^{l+2}\cap\mathbb{V}$.
Moreover if $f$ is an element of $\mathbb{H}$, then $v$ is unique and the following relations hold
\begin{eqnarray}\label{Eq GS-01}
(v,g)_\mathbb{V}=(f,g),\ \forall g\in \mathbb{V},
\end{eqnarray}
\begin{eqnarray}\label{Eq GS-02}
\|v\|_{\mathbb{H}^{l+2}}\leq C\|f\|_\mathbb{H}.
\end{eqnarray}
\end{lem}

\vskip 0.3cm
Recall the Stokes operator defined by
\begin{eqnarray}\label{Eq Stoke}
Au=-\mathbb{P}\Delta u,\ \forall u\in D(A)=\mathbb{H}^2(\mathcal{O})\cap\mathbb{V},
\end{eqnarray}
here the mapping $\mathbb{P}:\mathbb{L}^2(\mathcal{O})\rightarrow\mathbb{H}$ is the usual Helmholtz-Leray projector.
Lemma \ref{Lem GS} implies that the operator $(I+\alpha A)^{-1}$ defines an isomorphism form $\mathbb{H}^l(\mathcal{O})\cap\mathbb{H}$
into $\mathbb{H}^{l+2}(\mathcal{O})\cap\mathbb{V}$ provided that $\mathcal{O}$ is of class $\mathcal{C}^l$, $l\geq1$. Moreover, the following
properties hold
\begin{eqnarray*}
& & ((I+\alpha A)^{-1}f,g)_\mathbb{V}=(f,g),\\
& & \|(I+\alpha A)^{-1}f\|_\mathbb{V}\leq C|f|,\ \ \ \forall f\in \mathbb{H}^l(\mathcal{O})\cap\mathbb{V},\ g\in\mathbb{V}.
\end{eqnarray*}
From these facts, $\widehat{A}=(I+\alpha A)^{-1}A$ defines a
continuous linear operator from $\mathbb{H}^l(\mathcal{O})\cap\mathbb{V}$ onto itself for $l\geq2$, and satisfies
$$
(\widehat{A}u,v)_\mathbb{V}=(Au,v)=((u,v)),\ \ \forall u\in\mathbb{W},\ v\in\mathbb{V}.
$$
Hence
$$
(\widehat{A}u,u)_\mathbb{V}=\|u\|, \ \ \ \forall u\in\mathbb{W}.
$$

Let
$$
b(u,v,w)=((u\cdot \nabla)v, w)=\sum_{i,j=1}^2\int_{\mathcal{O}}u_i\frac{\partial v_j}{\partial x_i}w_jdx,\ \ \forall u,v,w\in\mathcal{C},
$$
By the incompressibility condition,
\begin{eqnarray}\label{eq b}
b(u,v,v)=0,\ \ \  \forall u,v\in\mathbb{V},
\end{eqnarray}
Moreover, the following identity holds (see for instance \cite{Bernard} \cite{CO}):
\begin{eqnarray}\label{curl}
((curl\Phi)\times v,w)=b(v,\Phi,w)-b(w,\Phi,v),
\end{eqnarray}
for any smooth function $\Phi,\ v$ and $w$. Now we recall the following estimates which can be found in \cite{RS-12}(Lemma 2.3 and Lemma 2.4),
 and also in  \cite{Bernard} \cite{CO}.

\begin{lem}\label{Lem B}
For any $u,v,w\in\mathbb{W}$, we have
\begin{eqnarray}\label{Ineq B 01}
    |(curl(u-\alpha\Delta u)\times v,w)|
\leq
    C\|u\|_{\mathbb{H}^3}\|v\|_\mathbb{V}\|w\|_{\mathbb{W}},
\end{eqnarray}
and
\begin{eqnarray}\label{Ineq B 02}
    |(curl(u-\alpha\Delta u)\times u,w)|
\leq
    C\|u\|^2_\mathbb{V}\|w\|_{\mathbb{W}}.
\end{eqnarray}
\end{lem}

Recall the definition of the bilinear operator $\widehat{B}(\cdot,\cdot):\ \mathbb{W}\times\mathbb{V}\rightarrow\mathbb{W}^*$ as
\begin{eqnarray}\label{Lem B-01}
\widehat{B}(u,v)=(I+\alpha A)^{-1}\Big( curl(u-\alpha \Delta u)\times v\Big).
\end{eqnarray}
\begin{lem}\label{Lem-B-01}
For any $u\in\mathbb{W}$ and $v\in\mathbb{V}$, it holds that
\begin{eqnarray}\label{Eq B-01}
    \|\widehat{B}(u,v)\|_{\mathbb{W}^*}
\leq
    C\|u\|_\mathbb{W}\|v\|_\mathbb{V},
\end{eqnarray}
and
\begin{eqnarray}\label{Eq B-02}
 \|\widehat{B}(u,u)\|_{\mathbb{W}^*}
\leq
    C_B\|u\|^2_\mathbb{V}.
\end{eqnarray}
In addition
\begin{eqnarray}\label{Eq B-03}
 \langle\widehat{B}(u,v),v\rangle=0,
\end{eqnarray}
which implies
\begin{eqnarray}\label{Eq B-04}
 \langle\widehat{B}(u,v),w\rangle=-\langle\widehat{B}(u,w),v\rangle,\ \ \ \forall u,\ v,\ w\in\mathbb{W}.
\end{eqnarray}

\end{lem}

\section{Hypotheses}

In this section, we will state the precise assumptions on the coefficients and collect some preliminary results from \cite{RS-10} and \cite{RS-12}, which will be used later.

\vskip 0.2cm
Let $F:\mathbb{V}\times[0,T]\rightarrow\mathbb{V}$
and $G:\mathbb{V}\times[0,T]\rightarrow\mathbb{V}^{\otimes m}$ be given measurable maps. We introduce the following conditions:

\vskip 0.2cm
{\bf (F1)}
\begin{eqnarray}\label{F-01}
            F(0,t)=0,
\end{eqnarray}
and
\begin{eqnarray}\label{F-02}
\|F(u_1,t)-F(u_2,t)\|_\mathbb{V}
\leq
C_1\|u_1-u_2\|_\mathbb{V},\ \ \forall u_1,u_2\in\mathbb{V},\ \ t\in[0,T].
\end{eqnarray}

{\bf (F2)} $F$ is differentiable with respect to the first variable, and the derivative $F':\mathbb{V}\times[0,T] \rightarrow L(\mathbb{V})$ is uniformly Lipschitz with respect to the first variable, more precisely, for any $t\in[0,T]$,
\begin{eqnarray}\label{F-03}
\|F'(u_1,t)-F'(u_2,t)\|_{L(\mathbb{V})}
\leq
C_2\|u_1-u_2\|_\mathbb{V},\ \ \forall u_1,u_2\in\mathbb{V},\ \ t\in[0,T].
\end{eqnarray}
By (\ref{F-02}), we conclude that
\begin{eqnarray}\label{F-04}
   \|F'(u,t)\|_{L(\mathbb{V})}\leq C_2.
\end{eqnarray}

{\bf (G)}
\begin{eqnarray}\label{G-01}
G(0,t)=0,
\end{eqnarray}
and
\begin{eqnarray}\label{G-02}
\|G(u_1,t)-G(u_2,t)\|_{\mathbb{V}^{\otimes m}}
\leq
C_3\|u_1-u_2\|_\mathbb{V},\ \ \forall u_1,u_2\in\mathbb{V},\ \ t\in[0,T].
\end{eqnarray}
Where $C_1,C_2,C_3$ are some constants independent of $u,t$.
Set
\begin{eqnarray*}
\widehat{F}'(u,t)=(I+\alpha A)^{-1} F'(u,t).
\end{eqnarray*}

{\bf Condition (F1)}, {\bf Condition (F2)} and {\bf Condition (G)} imply that there exist $C_F$, $C'_F$ and $C_G$
such that
\begin{eqnarray}\label{Jian F1}
\|\widehat{F}(u_1,t)-\widehat{F}(u_2,t)\|_\mathbb{V}\leq C_F\|u_1-u_2\|_\mathbb{V},
\end{eqnarray}
\begin{eqnarray}\label{Jian F2}
\|\widehat{F}'(u_1,t)-\widehat{F}'(u_2,t)\|_{L(\mathbb{V})}\leq C'_F\|u_1-u_2\|_\mathbb{V},
\end{eqnarray}
\begin{eqnarray}\label{Jian G}
\|\widehat{G}(u_1,t)-\widehat{G}(u_2,t)\|_{\mathbb{V}^{\otimes m}}\leq C_G\|u_1-u_2\|_\mathbb{V}.
\end{eqnarray}

Now we recall the concept of solution of the problem (\ref{01}).

\begin{dfn}\label{Def 01}
A $\mathbb{V}$-valued $\{\mathcal{F}_t\}$-adapted stochastic process $u^\varepsilon$ is called a solution of the system (\ref{01}), if the following two conditions hold

1. $u^\varepsilon\in L^p(\Omega,\mathcal{F},P;L^\infty([0,T],\mathbb{W}))\cap L^p(\Omega,\mathcal{F},P;C([0,T],\mathbb{V})),\ 2\leq p<\infty.$


2. For any $v\in\mathbb{W}$, the following identity holds $P$-a.s.
\begin{eqnarray*}
&  &(u^\varepsilon(t)-u^\varepsilon(0),v)_{\mathbb{V}}+\int_0^t[\nu ((u^\varepsilon(s),v))+\big(curl (u^\varepsilon(s)-\alpha \Delta u^\varepsilon(s)\big)\times u^\varepsilon(s),v)]ds\\
&=&
    \int_0^t\big(F(u^\varepsilon(s),s),v\big)ds+\sqrt{\varepsilon}\int_0^t\big(G(u^\varepsilon(s),s),v\big)dW(s),\ \ \forall t\in(0,T]
\end{eqnarray*}
 Or equivalently, $P$-a.s., the following equation
\begin{eqnarray*}
u^\varepsilon(t)+\int_0^t\Big(\nu \widehat{A}u^\varepsilon(s)+\widehat{B}(u^\varepsilon(s),u^\varepsilon(s))\Big)ds
=
u^\varepsilon(0)+\int_0^t\widehat{F}(u^\varepsilon(s),s)ds+\sqrt{\varepsilon}\int_0^t\widehat{G}(u^\varepsilon(s),s)dW(s),
\end{eqnarray*}
holds in $\mathbb{W}^*$ for any $t\in(0,T]$.
\end{dfn}

Applying Galerkin approximation schemes for the system (\ref{01}), Razafimandimby and Sango
obtained the following result (see Theorem 3.4 and Theorem 4.1 in \cite{RS-12}).
\begin{thm}\label{Solution Existence}
Let $u_0\in\mathbb{W}$. Assume conditions {\bf (F1)} and {\bf (G)} hold. Then the system (\ref{01}) (or equivalently the problem
(\ref{Abstract})) has a unique solution $u^\epsilon$. Moreover,
the solution $u^\varepsilon$ admits a version which is continuous in $\mathbb{W}$ with respect to
the weak topology.
\end{thm}

Recall the solution $u^0$ given in (\ref{Deterministic}). By Theorem 5.6 in \cite{CG}, we have the following lemma.
\begin{lem}\label{Regularity}
If we assume that the boundary $\partial \mathcal{O}$ is of class $\mathcal{C}^{3,1}$ and the initial value $u_0\in{\mathbb{V}\cap{\mathbb{H}}^4(\mathcal{O})}$, then $u^0$ belongs to $L^\infty([0,T],{\mathbb{V}\cap{\mathbb{H}}^4(\mathcal{O})})$, i.e.
\begin{eqnarray}\label{Regularity estimate}
\sup_{t\in[0,T]}\|u^0(t)\|_{{\mathbb{H}}^4(\mathcal{O})}\leq C.
\end{eqnarray}
\end{lem}

\vskip 0.2cm
To obtain the moderate deviation principles, additionally we impose the following hypothese throughout.

\vskip 0.2cm
{\bf (I)} the initial value $u_0\in{\mathbb{V}\cap{\mathbb{H}}^4(\mathcal{O})}$ and the boundary $\partial \mathcal{O}$ is of class $\mathcal{C}^{3,1}$.

\section{Central Limit Theorem}

\vskip 0.2cm
In this section, we will establish the central limit theorem. Let $u^{\varepsilon}$ and $u^{0}$ be the unique solution of (\ref{Abstract}) and (\ref{Deterministic}) respectively. The following estimates follow from Lemma 3.7 in \cite{RS-12}.
\begin{lem}
There exists a constant $\varepsilon_{0}>0$ such that, for any $2\leq p<\infty$
\begin{eqnarray}
\sup_{\varepsilon\in(0,\varepsilon_0)}E\Big[\sup_{s\in[0,T]}\|u^\varepsilon(s)\|^p_\mathbb{W}\Big]\leq C_p,\\
\sup_{s\in[0,T]}\|u^0(s)\|^p_\mathbb{W}\leq C_p.
\end{eqnarray}
\end{lem}


The next result is concerned with the convergence of $u^\varepsilon$ as $\varepsilon\rightarrow 0$.
\begin{prp}\label{CLT prp}
There exists a constant $\varepsilon_{0}>0$ such that, for all $0\leq\varepsilon\leq\varepsilon_{0}$,
\begin{eqnarray}\label{CLT prop}
E\Big[\sup_{s\in[0,T]}\|u^{\varepsilon}(s)-u^{0}(s)\|^p_{\mathbb{V}}\Big]\leq\varepsilon^{\frac{p}{2}}C_p,\ \text{for any }2\leq p<\infty,
\end{eqnarray}
\ \ where $C_p$ is a constant depending on $p$.
\end{prp}


\begin{proof}
For any integer $J \geq 1$ we introduce the stopping time 
$$\tau_{J}=\inf\{t\geq 0;\|u^{\varepsilon}(t)-u^{0}(t)\|_{\mathbb{V}}\geq J\}.$$

\vskip 0.2cm
Applying ${\rm It\hat{o}}$'s formula to $\|u^{\varepsilon}(t)-u^{0}(t)\|_{\mathbb{V}}^p$ for $p\geq2$, we have
\begin{eqnarray*}
& &d\|u^{\varepsilon}(t)-u^{0}(t)\|_{\mathbb{V}}^{p}
=d(\|u^{\varepsilon}(t)-u^{0}(t)\|_{\mathbb{V}}^{2})^{\frac{p}{2}}\nonumber\\
&=&\frac{p}{2}\|u^{\varepsilon}(t)-u^{0}(t)\|_{\mathbb{V}}^{p-2}\Big(-2\nu\|u^{\varepsilon}(t)-u^{0}(t)\|^{2}dt\nonumber\\
      & &-2\langle\widehat{B}(u^{\varepsilon}(t),u^{\varepsilon}(t))-\widehat{B}(u^{0}(t),u^{0}(t)),u^{\varepsilon}(t)-u^{0}(t)\rangle dt\nonumber\\
      & &+2\big(\widehat{F}(u^{\varepsilon}(t),t)-\widehat{F}(u^{0}(t),t),u^{\varepsilon}(t)-u^{0}(t)\big)_{\mathbb{V}}dt\nonumber\\
      & &+2\sqrt{\varepsilon}\big(\widehat{G}(u^{\varepsilon}(t),t),u^{\varepsilon}(t)-u^{0}(t)\big)_{\mathbb{V}}dW(t)
            +\varepsilon\|\widehat{G}(u^{\varepsilon}(t),t)\|^2_{\mathbb{V}^{\otimes m}}dt\Big)\nonumber\\
      & &+\frac{p}{4}(\frac{p}{2}-1)\|u^{\varepsilon}(t)-u^{0}(t)\|_{\mathbb{V}}^{p-4}4\varepsilon
            \big(\widehat{G}(u^{\varepsilon}(t),t),u^{\varepsilon}(t)-u^{0}(t)\big)_{\mathbb{V}}
            \big(\widehat{G}(u^{\varepsilon}(t),t),u^{\varepsilon}(t)-u^{0}(t)\big)_{\mathbb{V}}'dt.
\end{eqnarray*}


Integrating from $0$ to $t$, we have
\begin{eqnarray*}
& &\|u^{\varepsilon}(t)-u^{0}(t)\|_{\mathbb{V}}^{p}+
      p\nu\int_{0}^{t}\|u^{\varepsilon}(s)-u^{0}(s)\|_{\mathbb{V}}^{p-2}\|u^{\varepsilon}(s)-u^{0}(s)\|^{2}ds\nonumber\\
&=&-p\int_{0}^{t}\|u^{\varepsilon}(s)-u^{0}(s)\|_{\mathbb{V}}^{p-2}
      \langle\widehat{B}(u^{\varepsilon}(s),u^{\varepsilon}(s))-\widehat{B}(u^{0}(s),u^{0}(s)),u^{\varepsilon}(s)-u^{0}(s)\rangle ds\nonumber\\
& &+p\int_{0}^{t}\|u^{\varepsilon}(s)-u^{0}(s)\|_{\mathbb{V}}^{p-2}
      \big(\widehat{F}(u^{\varepsilon}(s),s)-\widehat{F}(u^{0}(s),s),u^{\varepsilon}(s)-u^{0}(s)\big)_{\mathbb{V}}ds\nonumber\\
& &+\sqrt{\varepsilon}p\int_{0}^{t}\|u^{\varepsilon}(s)-u^{0}(s)\|_{\mathbb{V}}^{p-2}
      \big(\widehat{G}(u^{\varepsilon}(s),s),u^{\varepsilon}(s)-u^{0}(s)\big)_{\mathbb{V}}dW(s)\\
& &+\frac{p}{2}\varepsilon\int_{0}^{t}\|u^{\varepsilon}(s)-u^{0}(s)\|_{\mathbb{V}}^{p-2}
      \|\widehat{G}(u^{\varepsilon}(s),s)\|^2_{\mathbb{V}^{\otimes m}}ds\nonumber\\
& &+p(\frac{p}{2}-1)\varepsilon\int_{0}^{t}\|u^{\varepsilon}(s)-u^{0}(s)\|_{\mathbb{V}}^{p-4}
      \big(\widehat{G}(u^{\varepsilon}(s),s),u^{\varepsilon}(s)-u^{0}(s)\big)_{\mathbb{V}}
      \big(\widehat{G}(u^{\varepsilon}(s),s),u^{\varepsilon}(s)-u^{0}(s)\big)_{\mathbb{V}}'ds.\nonumber
\end{eqnarray*}


Taking sup over the interval $[0,T\wedge\tau_{J}]$ and taking expectation,
\begin{eqnarray}\label{CLT lemma 00}
& &E\Big[\sup_{t\in[0,T\wedge\tau_{J}]}\|u^{\varepsilon}(t)-u^{0}(t)\|_{\mathbb{V}}^{p}\Big]+
              \nu p\ E\Big[\int_{0}^{T\wedge\tau_{J}}\|u^{\varepsilon}(s)-u^{0}(s)\|_{\mathbb{V}}^{p-2}\|u^{\varepsilon}(s)-u^{0}(s)\|^{2}ds)\Big]\nonumber\\
&\leq&
pE\Big[\int_{0}^{T\wedge\tau_{J}}\|u^{\varepsilon}(s)-u^{0}(s)\|_{\mathbb{V}}^{p-2}\big|\big\langle\widehat{B}(u^{\varepsilon}(s),u^{\varepsilon}(s))-
              \widehat{B}(u^{0}(s),u^{0}(s)),u^{\varepsilon}(s)-u^{0}(s)\big\rangle\big|ds\Big]\nonumber\\
& &+pE\Big[\int_{0}^{T\wedge\tau_{J}}\|u^{\varepsilon}(s)-u^{0}(s)\|_{\mathbb{V}}^{p-2}
              \big|\big(\widehat{F}(u^{\varepsilon}(s),s)-\widehat{F}(u^{0}(s),s),u^{\varepsilon}(s)-u^{0}(s)\big)_{\mathbb{V}}\big|ds\Big]\nonumber\\
& &+\sqrt{\varepsilon}pE\Big[\sup_{0\leq t\leq T\wedge\tau_{J}}\Big|\int_{0}^{t}\|u^{\varepsilon}(s)-u^{0}(s)\|_{\mathbb{V}}^{p-2}
              \big(\widehat{G}(u^{\varepsilon}(t),t),u^{\varepsilon}(s)-u^{0}(s)\big)_{\mathbb{V}}dW(s)\Big|\Big]\nonumber\\
& &+\frac{p}{2}\varepsilon E\Big[\int_{0}^{T\wedge\tau_{J}}\|u^{\varepsilon}(s)-u^{0}(s)\|_{\mathbb{V}}^{p-2}
              \|\widehat{G}(u^{\varepsilon}(s),s)\|^2_{\mathbb{V}^{\otimes m}}ds\Big]\nonumber\\
& &+p(\frac{p}{2}-1)\varepsilon E\Big[\int_{0}^{T\wedge\tau_{J}}\|u^{\varepsilon}(s)-u^{0}(s)\|_{\mathbb{V}}^{p-4}\nonumber\\
& &\qquad\big(\widehat{G}(u^{\varepsilon}(s),s),u^{\varepsilon}(s)-u^{0}(s)\big)_{\mathbb{V}}
         \big(\widehat{G}(u^{\varepsilon}(s),s),u^{\varepsilon}(s)-u^{0}(s)\big)_{\mathbb{V}}'ds\Big]\nonumber\\
& &:=I_{1}(T)+I_{2}(T)+I_{3}(T)+I_{4}(T)+I_{5}(T).
\end{eqnarray}

Lemma \ref{Lem B} and Lemma \ref{Lem-B-01} imply
\begin{eqnarray}\label{CLT lemma 01}
I_{1}(T)
&\leq&
C_{p}E\Big[\int_{0}^{T\wedge\tau_{J}}\|u^{\varepsilon}(s)-u^{0}(s)\|_{\mathbb{V}}^{p-2}\big|\big\langle\widehat{B}\big(u^{\varepsilon}(s)-u^{0}(s)
              ,u^{\varepsilon}(s)\big),u^{\varepsilon}(s)-u^{0}(s)\big\rangle\big|ds\Big]\nonumber\\
&\leq&
C_{p}E\Big[\int_{0}^{T\wedge\tau_{J}}\|u^{\varepsilon}(s)-u^{0}(s)\|_{\mathbb{V}}^{p-2}\|u^{\varepsilon}(s)-u^{0}(s)\|_{\mathbb{V}}^{2}
              \|u^{0}(s)\|_{\mathbb{W}}ds\Big]\nonumber\\
&\leq&
C_{p}\int_0^{T}E\Big[\sup_{r\in[0,s\wedge\tau_{J}]}(\|u^{\varepsilon}(r)-u^{0}(r)\|_{\mathbb{V}}^{p})\Big]\|u^0(s)\|_{\mathbb{W}}ds.
\end{eqnarray}

\vskip 0.2cm

By (\ref{Jian F1}), we have
\begin{eqnarray}\label{CLT lemma 02}
I_2(T)
&\leq&
C_{p}E\Big[\int_{0}^{T\wedge\tau_{J}}\|u^{\varepsilon}(s)-u^{0}(s)\|_{\mathbb{V}}^{p-2}\|u^{\varepsilon}(s)-u^{0}(s)\|_{\mathbb{V}}^{2}ds\Big]\nonumber\\
&\leq&
C_{p}\int_0^{T\wedge\tau_{J}}E\Big[\sup_{r\in[0,s\wedge\tau_{J}]}\|u^{\varepsilon}(r)-u^{0}(r)\|_{\mathbb{V}}^{p}\Big]ds.
\end{eqnarray}


Applying (\ref{Jian G}) and the B-D-G and Young's inequalities to $I_3$, we have for any $\delta>0$
\begin{eqnarray}\label{CLT lemma 03}
I_3(T)
&\leq&
\sqrt{\varepsilon}C_{p}E\Big[\int_0^{T\wedge\tau_{J}}\|u^{\varepsilon}(s)-u^{0}(s)\|_{\mathbb{V}}^{2p-2}
      \|u^{\varepsilon}(s)\|_{\mathbb{V}}^{2}ds\Big]^{\frac{1}{2}}\nonumber\\
&\leq&
\sqrt{\varepsilon}C_{p}E\Big[\sup_{s\in[0,T\wedge\tau_{J}]}\|u^{\varepsilon}(s)-u^{0}(s)\|_{\mathbb{V}}^{\frac{p}{2}}
      \big(\int_0^{T\wedge\tau_{J}}\|u^{\varepsilon}(s)-u^{0}(s)\|_{\mathbb{V}}^{p-2}\|u^{\varepsilon}(s)\|_{\mathbb{V}}^{2}ds\big)^{\frac{1}{2}}\Big]\nonumber\\
&\leq&
\delta E\Big[\sup_{s\in[0,T\wedge\tau_{J}]}\|u^{\varepsilon}(s)-u^{0}(s)\|_{\mathbb{V}}^{p}\Big]
      +\varepsilon C_{p,\delta}E\Big[\int_{0}^{T\wedge\tau_{J}}\|u^{\varepsilon}(s)-u^{0}(s)\|_{\mathbb{V}}^{p-2}\|u^{\varepsilon}(s)\|_{\mathbb{V}}^{2}ds\Big]\nonumber\\
&\leq&
\delta E\Big[\sup_{s\in[0,T\wedge\tau_{J}]}\|u^{\varepsilon}(s)-u^{0}(s)\|_{\mathbb{V}}^{p}\Big]
      +C_{p,\delta}E\Big[\int_{0}^{T\wedge\tau_{J}}\|u^{\varepsilon}(s)-u^{0}(s)\|_{\mathbb{V}}^{p}ds\Big]\nonumber\\
& &\quad+\varepsilon^{\frac{p}{2}}C_{p,\delta}E\Big[\int_{0}^{T\wedge\tau_{J}}\|u^{\varepsilon}(s)\|_{\mathbb{V}}^{p}ds\Big]\nonumber\\
&\leq&
\delta E\Big[\sup_{s\in[0,T\wedge\tau_{J}]}\|u^{\varepsilon}(s)-u^{0}(s)\|_{\mathbb{V}}^{p}\Big]
      +C_{p,\delta}\int_0^{T}E\Big[\sup_{r\in[0,s\wedge\tau_{J}]}\|u^{\varepsilon}(r)-u^{0}(r)\|_{\mathbb{V}}^{p}\Big]ds\nonumber\\
& &\quad+\varepsilon^{\frac{p}{2}}C_{p,\delta,T}E\Big[\sup_{s\in[0,T]}\|u^{\varepsilon}(s)\|_{\mathbb{V}}^{p}\Big].
\end{eqnarray}


By (\ref{G-01}) and (\ref{Jian G}), we have
\begin{eqnarray}\label{CLT lemma 04}
I_4(T)
&\leq&
\varepsilon C_{p}E\Big[\int_{0}^{T\wedge\tau_{J}}\|u^{\varepsilon}(s)-u^{0}(s)\|_{\mathbb{V}}^{p-2}\|u^{\varepsilon}(s)\|_{\mathbb{V}}^{2}ds\Big]\nonumber\\
&\leq&
C_{p}\int_0^{T}E\Big[\sup_{r\in[0,s\wedge\tau_{J}]}\|u^{\varepsilon}(r)-u^{0}(r)\|_{\mathbb{V}}^{p}\Big]ds
      +\varepsilon^{\frac{p}{2}}C_{p,T}E\Big[\sup_{s\in[0,T]}\|u^{\varepsilon}(s)\|_{\mathbb{V}}^{p}\Big],
\end{eqnarray}
and
\begin{eqnarray}\label{CLT lemma 05}
I_5(T)
&\leq&
\varepsilon C_{p}E\Big[\int_{0}^{T\wedge\tau_{J}}\|u^{\varepsilon}(s)-u^{0}(s)\|_{\mathbb{V}}^{p-2}\|u^{\varepsilon}(s)\|_{\mathbb{V}}^{2}\Big]\nonumber\\
&\leq&
C_{p}\int_0^{T}E\Big[\sup_{r\in[0,s\wedge\tau_{J}]}\|u^{\varepsilon}(r)-u^{0}(r)\|_{\mathbb{V}}^{p}\Big]ds
      +\varepsilon^{\frac{p}{2}}C_{p,T}E\Big[\sup_{s\in[0,T]}\|u^{\varepsilon}(s)\|_{\mathbb{V}}^{p}\Big].
\end{eqnarray}


Combining (\ref{CLT lemma 00})--(\ref{CLT lemma 05}), we get
\begin{eqnarray}
& &(1-\delta)E\Big[\sup_{s\in[0,T\wedge\tau_{J}]}\|u^{\varepsilon}(s)-u^{0}(s)\|_{\mathbb{V}}^{p}\Big]\nonumber\\
& &\quad+\nu pE\Big[\int_{0}^{T\wedge\tau_{J}}\|u^{\varepsilon}(s)-u^{0}(s)\|_{\mathbb{V}}^{p-2}\|u^{\varepsilon}(s)-u^{0}(s)\|^{2}ds\Big]\nonumber\\
&\leq&
C_{p}\int_0^{T}E\Big[\sup_{r\in[0,s\wedge\tau_{J}]}(\|u^{\varepsilon}(r)-u^{0}(r)\|_{\mathbb{V}}^{p})\Big]\big(\|u^0(s)\|_{\mathbb{W}}+1\big)ds\nonumber\\
& &+\varepsilon^{\frac{p}{2}}C_{p,\delta,T}E\Big[\sup_{s\in[0,T]}\|u^{\varepsilon}(s)\|_{\mathbb{V}}^{p}\Big].
\end{eqnarray}
Choosing $\delta=\frac{1}{2}$ and applying Gronwall's Inequality, we obtain
\begin{eqnarray}
E\Big[\sup_{s\in[0,T\wedge\tau_{J}]}\|u^\varepsilon(s)-u^{0}(s)\|^p_\mathbb{V}\Big]\leq\varepsilon^{\frac{p}{2}}C_p.
\end{eqnarray}
Finally, let $J\rightarrow\infty$ to obtain (\ref{CLT prop}).
\end{proof}


\vskip 0.3cm
Consider the following SPDE:
\begin{eqnarray}\label{zheng}
& &dV^{0}(t)+\nu\widehat{A}V^{0}(t)dt\nonumber\\
&=&-\widehat{B}(V^{0}(t),u^{0}(t))-\widehat{B}(u^0(t),V^0(t))+\widehat{F}'(u^0(t),t)V^0(t)dt+\widehat{G}(u^0(t),t)dW(t).
\end{eqnarray}
with initial value $V^0(0)=0$. Using Galerkin approximations and Lemma \ref{Regularity} as in \cite{RS-12}, we can prove that there exists a unique solution of (\ref{zheng}), and we have the following estimate:
\begin{eqnarray}
E\Big[\sup_{s\in[0,T]}\|V^0(s)\|^p_\mathbb{W}\Big]\leq C_p\nonumber,\ \ \ p\geq 2.
\end{eqnarray}
Actually, the details of a prior estimates of Galerkin approximations are also similar to the proof of Lemma \ref{MDP lemma 01} below.

\vskip 0.2cm
Our first main result is the following central limit theorem.
\begin{thm}
Set $V^{\varepsilon}(t):=\frac{u^\varepsilon-u^0}{\sqrt{\varepsilon}}$.\\
Then we have
\begin{eqnarray}\label{CLT limit}
\lim_{\varepsilon\rightarrow0}E\Big[\sup_{s\in[0,T]}\|V^{\varepsilon}(s)-V^0(s)\|_\mathbb{V}^p\Big]=0.
\end{eqnarray}
\end{thm}


\begin{proof}
Noticing that $V^{\varepsilon}(0)=0$,
\begin{eqnarray}
&  &dV^{\varepsilon}(t)+\nu\widehat{A}V^{\varepsilon}(t)dt+
\widehat{B}(V^{\varepsilon}(t),u^{\varepsilon}(t))dt+\widehat{B}(u^{0}(t),V^{\varepsilon}(t))dt\nonumber\\
&=&
\frac{1}{\sqrt{\varepsilon}}\big(\widehat{F}(u^{\varepsilon}(t),t)-\widehat{F}(u^{0}(t),t)\big)dt+\widehat{G}(u^{\varepsilon}(t),t)dW(t),\nonumber
\end{eqnarray}
We have
\begin{eqnarray*}
& &d(V^{\varepsilon}(t)-V^0(t))+\nu\widehat{A}(V^{\varepsilon}(t)-V^0(t))dt\nonumber\\
& &+(\widehat{B}(V^{\varepsilon}(t),u^{\varepsilon}(t))-\widehat{B}(V^{0}(t),u^{0}(t)))dt
      +(\widehat{B}(u^{0}(t),V^{\varepsilon}(t))-\widehat{B}(u^0(t),V^0(t)))dt\nonumber\\
&=&
\frac{1}{\sqrt{\varepsilon}}\big(\widehat{F}(u^{\varepsilon}(t),t)-\widehat{F}(u^{0}(t),t)\big)dt-\widehat{F}'(u^0(t),t)V^0(t)dt\nonumber\\
& &\quad+\big(\widehat{G}(u^{\varepsilon}(t),t)-\widehat{G}(u^{0}(t),t)\big)dW(t).
\end{eqnarray*}

By Proposition \ref{CLT prp}, we have
\begin{eqnarray}
E\Big[\sup_{s\in[0,T]}\|V^{\varepsilon}(s)\|^p_\mathbb{V}\Big]\leq C_{p}.
\end{eqnarray}
Now applying ${\rm It\hat{o}}$'s formula to $\|V^{\varepsilon}(t)-V^{0}(t)\|_{\mathbb{V}}^p$ for $p\geq2$, we have
\begin{eqnarray*}
& &d\|V^{\varepsilon}(t)-V^{0}(t)\|_{\mathbb{V}}^{p}
=d(\|V^{\varepsilon}(t)-V^{0}(t)\|_{\mathbb{V}}^{2})^{\frac{p}{2}}\nonumber\\
&=&
\frac{p}{2}\|V^{\varepsilon}(t)-V^{0}(t)\|_{\mathbb{V}}^{p-2}
    \Big(
         -2\nu\|V^{\varepsilon}(t)-V^{0}(t)\|^{2}\nonumber\\
         & &\quad-2\langle\widehat{B}(V^{\varepsilon}(t),u^{\varepsilon}(t))-\widehat{B}(V^{0}(t),u^{0}(t)),V^{\varepsilon}(t)-V^0(t)\rangle dt\nonumber\\
         & &\quad+2\big(\frac{1}{\sqrt{\varepsilon}}\big(\widehat{F}(u^{\varepsilon}(t),t)-\widehat{F}(u^0(t),t)\big)-\widehat{F}'(u^0(t),t)V^0(t),
            V^{\varepsilon}(t)-V^0(t)\big)_{\mathbb{V}}dt\nonumber\\
         & &\quad+2\big(\widehat{G}(u^\varepsilon(t),t)-\widehat{G}(u^0(t),t),V^{\varepsilon}(t)-V^0(t)\big)_{\mathbb{V}}dW(t)
         +\|\widehat{G}(u^\varepsilon(t),t)-\widehat{G}(u^0(t),t)\|_{\mathbb{V}^{\otimes m}}^2dt
    \Big)\nonumber\\
& &+p(\frac{p}{2}-1)\|V^{\varepsilon}(t)-V^{0}(t)\|_{\mathbb{V}}^{p-4}\big(\widehat{G}(u^\varepsilon(t),t)-\widehat{G}(u^0(t),t),
    V^{\varepsilon}(t)-V^0(t)\big)_{\mathbb{V}}\nonumber\\
& &\qquad\big(\widehat{G}(u^\varepsilon(t),t)-\widehat{G}(u^0(t),t),V^{\varepsilon}(t)-V^0(t)\big)'_{\mathbb{V}}dt.
\end{eqnarray*}

Integrating from $0$ to $t$,
\begin{eqnarray*}
& &\|V^{\varepsilon}(t)-V^{0}(t)\|_{\mathbb{V}}^{p}+
\nu p\int_0^{t}\|V^{\varepsilon}(s)-V^{0}(s)\|_{\mathbb{V}}^{p-2}\|V^{\varepsilon}(s)-V^{0}(s)\|^{2}ds\nonumber\\
&=&-p\int_0^{t}\|V^{\varepsilon}(s)-V^{0}(s)\|_{\mathbb{V}}^{p-2}
      \langle\widehat{B}(V^{\varepsilon}(s),u^{\varepsilon}(s))-\widehat{B}(V^{0}(s),u^{0}(s)),V^{\varepsilon}(s)-V^0(s)\rangle ds\nonumber\\
& &+p\int_0^{t}\|V^{\varepsilon}(s)-V^{0}(s)\|_{\mathbb{V}}^{p-2}\nonumber\\
      & &\qquad\big(\frac{1}{\sqrt{\varepsilon}}(\widehat{F}(u^{\varepsilon}(s),s)-F(u^0(s),s))-\widehat{F}'(u^0(s),s)V^0(s),
      V^{\varepsilon}(s)-V^0(s)\big)_{\mathbb{V}}ds\nonumber\\
& &+p\int_0^{t}\|V^{\varepsilon}(s)-V^{0}(s)\|_{\mathbb{V}}^{p-2}
      \big(\widehat{G}(u^\varepsilon(s),s)-\widehat{G}(u^0(s),s),V^{\varepsilon}(s)-V^0(s)\big)_{\mathbb{V}}dW(s)\nonumber\\
& &+\frac{p}{2}\int_0^{t}\|V^{\varepsilon}(s)-V^{0}(s)\|_{\mathbb{V}}^{p-2}
      \|\widehat{G}(u^\varepsilon(s),s)-\widehat{G}(u^0(s),s)\|_{\mathbb{V}^{\otimes m}}^2ds\nonumber\\
& &+p(\frac{p}{2}-1)\int_0^{t}\|V^{\varepsilon}(s)-V^{0}(s)\|_{\mathbb{V}}^{p-4}\big(\widehat{G}(u^\varepsilon(s),s)-\widehat{G}(u^0(s),s),
V^{\varepsilon}(s)-V^0(s)\big)_{\mathbb{V}}\nonumber\\
& &\qquad\big(\widehat{G}(u^\varepsilon(s),s)-\widehat{G}(u^0(s),s),V^{\varepsilon}(s)-V^0(s)\big)_{\mathbb{V}}'ds.
\end{eqnarray*}
Define $\tau_{J}=\inf\{t\geq0;\|V^{\varepsilon}(t)-V^0(t)\|^p_\mathbb{V}\geq J\}$. Taking sup over the interval $[0,T\wedge\tau_{J}]$ and taking expectation,
\begin{eqnarray}\label{CLT 00}
& &E\Big[\sup_{s\in[0,T\wedge\tau_{J}]}\|V^\varepsilon(s)-V^{0}(s)\|^p_\mathbb{V}\Big]
+\nu pE\Big[\int_0^{T\wedge\tau_{J}}\|V^{\varepsilon}(s)-V^{0}(s)\|_{\mathbb{V}}^{p-2}\|V^{\varepsilon}(s)-V^{0}(s)\|^{2}ds\Big]\nonumber\\
&\leq&
pE\Big[\int_0^{T\wedge\tau_{J}}\|V^{\varepsilon}(s)-V^{0}(s)\|_{\mathbb{V}}^{p-2}
      \big|\langle\widehat{B}(V^{\varepsilon}(s),u^{\varepsilon}(s))-\widehat{B}(V^{0}(s),u^{0}(s)),V^{\varepsilon}(s)-V^0(s)\rangle\big|ds\Big]\nonumber\\
& &+pE\Big[\int_0^{T\wedge\tau_{J}}\|V^{\varepsilon}(s)-V^{0}(s)\|_{\mathbb{V}}^{p-2}\nonumber\\
& &\qquad\big|\big(\frac{1}{\sqrt{\varepsilon}}\big(\widehat{F}(u^{\varepsilon}(s),s)-\widehat{F}(u^0(s),s)\big)-\widehat{F}'(u^0(s),s)V^0(s),
       V^{\varepsilon}(s)-V^0(s)\big)_{\mathbb{V}}\big|ds\Big]\nonumber\\
& &+pE\Big[\sup_{t\in[0,T\wedge\tau_{J}]}\Big|\int_0^{t}\|V^{\varepsilon}(s)-V^{0}(s)\|_{\mathbb{V}}^{p-2}
      2\big(\widehat{G}(u^\varepsilon(s),s)-\widehat{G}(u^0(s),s),V^{\varepsilon}(s)-V^0(s)\big)_{\mathbb{V}}dW(s)\Big|\Big]\nonumber\\
& &+\frac{p}{2}E\Big[\int_0^{T\wedge\tau_{J}}\|V^{\varepsilon}(s)-V^{0}(s)\|_{\mathbb{V}}^{p-2}
      \|\widehat{G}(u^\varepsilon(s),s)-\widehat{G}(u^0(s),s)\|_{\mathbb{V}^{\otimes m}}^2ds\Big]\nonumber\\
& &+p(\frac{p}{2}-1)E\Big[\int_0^{T\wedge\tau_{J}}\|V^{\varepsilon}(s)-V^{0}(s)\|_{\mathbb{V}}^{p-4}
      \big(\widehat{G}(u^\varepsilon(s),s)-\widehat{G}(u^0(s),s),V^{\varepsilon}(s)-V^0(s)\big)_{\mathbb{V}}\nonumber\\
& &\qquad\big(\widehat{G}(u^\varepsilon(s),s)-\widehat{G}(u^0(s),s),V^{\varepsilon}(s)-V^0(s)\big)_{\mathbb{V}}'ds\Big]\nonumber\\
&:=&
J_{1}(T)+J_{2}(T)+J_{3}(T)+J_{4}(T)+J_{5}(T).
\end{eqnarray}

By Lemma \ref{Lem-B-01},
\begin{eqnarray*}
& &\big|\langle\widehat{B}\big(V^{\varepsilon}(s),u^{\varepsilon}(s)\big)
-\widehat{B}\big(V^{0}(s),u^{0}(s)\big),V^{\varepsilon}(s)-V^0(s)\rangle\big|\nonumber\\
&=&\big|-\langle\widehat{B}\big(V^{\varepsilon}(s)-V^0(s),V^{\varepsilon}(s)-V^0(s)\big),u^0(s)\rangle
+\sqrt{\varepsilon}\langle\widehat{B}\big(V^{\varepsilon}(s),V^{\varepsilon}(s)\big),V^0(s)\rangle\big|\nonumber\\
&\leq&
C\|V^{\varepsilon}(s)-V^{0}(s)\|_{\mathbb{V}}^{2}\|u^0(s)\|_{\mathbb{W}}
+\sqrt{\varepsilon}C\|V^{\varepsilon}(s)\|_{\mathbb{V}}^{2}\|V^0(s)\|_{\mathbb{W}}.
\end{eqnarray*}
Substituting the above inequality into $J_1$, we get for $\delta>0$
\begin{eqnarray}\label{CLT 01}
J_{1}(T)
&\leq&
C_{p}E\Big[\int_0^{T\wedge\tau_{J}}\|V^{\varepsilon}(s)-V^{0}(s)\|_{\mathbb{V}}^{p}\|u^0(s)\|_{\mathbb{W}}\nonumber\\
& &\quad+\sqrt{\varepsilon}C\|V^{\varepsilon}(s)-V^{0}(s)\|_{\mathbb{V}}^{p-2}\|V^{\varepsilon}(s)\|_{\mathbb{V}}^{2}\|V^0(s)\|_{\mathbb{W}}ds\Big]\nonumber\\
&\leq&
C_{p}\int_0^{T}E\Big[\sup_{r\in[0,s\wedge\tau_{J}]}\|V^{\varepsilon}(r)-V^{0}(r)\|_{\mathbb{V}}^{p}\Big]\|u^0(s)\|_{\mathbb{W}}ds\nonumber\\
& &\quad+\sqrt{\varepsilon}C_{p}E\Big[\int_0^{T\wedge\tau_{J}}
      \big(
          \delta\|V^{\varepsilon}(s)-V^{0}(s)\|_{\mathbb{V}}^{p}
          +C_{\delta}(\|V^{\varepsilon}(s)\|_{\mathbb{V}}^{2}\|V^0(s)\|_{\mathbb{W}})^{\frac{p}{2}}
      \big)ds\Big]\nonumber\\
&\leq&
C_{p,\delta}\int_0^{T}E\Big[\sup_{r\in[0,s\wedge\tau_{J}]}\|V^{\varepsilon}(r)-V^{0}(r)\|_{\mathbb{V}}^{p}\Big]\big(\|u^0(s)\|_{\mathbb{W}}+1\big)ds\nonumber\\
& &\quad+\sqrt{\varepsilon}C_{p,\delta}
      \Big(
            E\Big[\int_0^{T}\|V^{\varepsilon}(s)\|_{\mathbb{V}}^{2p}ds\Big]
            +E\Big[\int_0^{T}\|V^0(s)\|_{\mathbb{W}}^pds\Big]
      \Big)\nonumber\\
&\leq&
C_{p,\delta}\int_0^{T}E\Big[\sup_{r\in[0,s\wedge\tau_{J}]}\|V^{\varepsilon}(r)-V^{0}(r)\|_{\mathbb{V}}^{p}\Big]\big(\|u^0(s)\|_{\mathbb{W}}+1\big)ds\nonumber\\
& &\quad+\sqrt{\varepsilon}C_{p,\delta,T}
      \Big(E\Big[\sup_{s\in[0,T]}\|V^{\varepsilon}(s)\|_{\mathbb{V}}^{2p}\Big]
            +E\Big[\sup_{s\in[0,T]}\|V^0(s)\|_{\mathbb{W}}^p\Big]\Big).
\end{eqnarray}

By (\ref{F-02}) and (\ref{Jian F2}), we have
\begin{eqnarray*}
& &\big|\big(\frac{1}{\sqrt{\varepsilon}}\big(\widehat{F}(u^{\varepsilon}(s),s)-\widehat{F}(u^0(s),s)\big)-\widehat{F}'(u^0(s),s)V^0(s),
V^{\varepsilon}(s)-V^0(s)\big)_{\mathbb{V}}\big|\nonumber\\
&\leq&
    \Big|\Big(\frac{1}{\sqrt{\varepsilon}}\big(F(u^{\varepsilon}(s),s)-F(u^0(s),s)\big)-F'(u^0(s),s)\frac{u^{\varepsilon}(s)-u^0(s)}{\sqrt{\varepsilon}}
          ,V^{\varepsilon}(s)-V^0(s)\Big)_{\mathbb{V}}\Big|\nonumber\\
& &+|\big(F'(u^0(s),s)\frac{u^{\varepsilon}(s)-u^0(s)}{\sqrt{\varepsilon}}-F'(u^0(s),s)V^0(s),V^{\varepsilon}(s)-V^0(s)\big)_{\mathbb{V}}|\nonumber\\
&\leq&
C\|\theta(s)\big(u^{\varepsilon}(s)-u^{0}(s)\big)\|_{\mathbb{V}}\|V^{\varepsilon}(s)-V^0(s)\|_\mathbb{V}\|V^{\varepsilon}(s)\|_\mathbb{V}
      +C\|V^{\varepsilon}(s)-V^0(s)\|_\mathbb{V}^2\nonumber\\
&\leq&
\sqrt{\varepsilon}C\|V^{\varepsilon}(s)\|^2_\mathbb{V}\|V^{\varepsilon}(s)-V^0(s)\|_\mathbb{V}
      +C\|V^{\varepsilon}(s)-V^0(s)\|_\mathbb{V}^2.
\end{eqnarray*}
$\theta(s)\in(0,1)$ comes from the mean value theorem.\\
Thus, we have
\begin{eqnarray}\label{CLT 02}
& &J_{2}(T)
\leq
C_{p}E\Big[\int_0^{T\wedge\tau_{J}}\|V^{\varepsilon}(s)-V^{0}(s)\|_{\mathbb{V}}^{p-2}\nonumber\\
& &\qquad\quad\quad\quad\quad\quad\quad\big(\sqrt{\varepsilon}\|V^{\varepsilon}(s)\|^2_\mathbb{V}\|V^{\varepsilon}(s)-V^0(s)\|_\mathbb{V}
      +\|V^{\varepsilon}(s)-V^0(s)\|_\mathbb{V}^2\big)ds\Big]\nonumber\\
&\leq&
C_{p}E\Big[\int_0^{T\wedge\tau_{J}}\|V^{\varepsilon}(s)-V^{0}(s)\|_{\mathbb{V}}^{p}ds\Big]
+\sqrt{\varepsilon}C_{p}E\Big[\int_0^{T\wedge\tau_{J}}\|V^{\varepsilon}(s)-V^{0}(s)\|_{\mathbb{V}}^{p-1}\|V^{\varepsilon}(s)\|^2_\mathbb{V}ds\Big]\nonumber\\
&\leq&
C_{p}\int_0^{T}E\Big[\sup_{s\in[0,t\wedge\tau_{J}]}\|V^{\varepsilon}(s)-V^{0}(s)\|_{\mathbb{V}}^{p}\Big]ds
      +\sqrt{\varepsilon}C_{p,T}E\Big[\sup_{s\in[0,T]}\|V^{\varepsilon}(s)\|_{\mathbb{V}}^{2p}\Big].
\end{eqnarray}

By the B-D-G inequality and (\ref{Jian G}),
\begin{eqnarray}\label{CLT 03}
& &J_{3}(T)
\leq
C_{p}E\Big[\int_0^{T\wedge\tau_{J}}\|V^{\varepsilon}(s)-V^{0}(s)\|_{\mathbb{V}}^{2p-2}\|u^{\varepsilon}(s)-u^{0}(s)\|_{\mathbb{V}}^{2}ds\Big]^{\frac{1}{2}}\nonumber\\
&\leq&
C_{p}E\Big[\sup_{s\in[0,T\wedge\tau_{J}]}\|V^{\varepsilon}(s)-V^{0}(s)\|_{\mathbb{V}}^{\frac{p}{2}}\big(\int_0^{T\wedge\tau_{J}}
       \|V^{\varepsilon}(s)-V^{0}(s)\|_{\mathbb{V}}^{p-2}\|u^{\varepsilon}(s)-u^{0}(s)\|_{\mathbb{V}}^{2}ds\big)^{\frac{1}{2}}\Big]\nonumber\\
&\leq&
\delta E\Big[\sup_{s\in[0,T\wedge\tau_{J}]}\|V^{\varepsilon}(s)-V^{0}(s)\|_{\mathbb{V}}^{p}\Big]\nonumber\\
& &\quad+C_{p,\delta}E\Big[\int_0^{T\wedge\tau_{J}}\|V^{\varepsilon}(s)-V^{0}(s)\|_{\mathbb{V}}^{p-2}
       \|u^{\varepsilon}(s)-u^{0}(s)\|_{\mathbb{V}}^{2}ds\Big]\nonumber\\
&\leq&
\delta E\Big[\sup_{s\in[0,T\wedge\tau_{J}]}\|V^{\varepsilon}(s)-V^{0}(s)\|_{\mathbb{V}}^{p}\Big]
      +C_{p,\delta}E\Big[\int_0^{T\wedge\tau_{J}}\|V^{\varepsilon}(s)-V^{0}(s)\|_{\mathbb{V}}^{p}ds\Big]\nonumber\\
& &\quad+\varepsilon^{\frac{p}{2}}C_{p,\delta}E\Big[\int_0^{T\wedge\tau_{J}}\|V^{\varepsilon}(s)\|_{\mathbb{V}}^{p}ds\Big]\nonumber\\
&\leq&
\delta E\Big[\sup_{s\in[0,T\wedge\tau_{J}]}\|V^{\varepsilon}(s)-V^{0}(s)\|_{\mathbb{V}}^{p}\Big]
      +C_{p,\delta}\int_0^{T}E\Big[\sup_{r\in[0,s\wedge\tau_{J}]}\|V^{\varepsilon}(r)-V^{0}(r)\|_{\mathbb{V}}^{p}\Big]ds\nonumber\\
& &\quad+\varepsilon^{\frac{p}{2}}C_{p,\delta,T}E\Big[\sup_{s\in[0,T]}\|V^{\varepsilon}(s)\|_{\mathbb{V}}^{p}\Big].
\end{eqnarray}

Using (\ref{Jian G}) again, we have
\begin{eqnarray}\label{CLT 04}
& &J_{4}(T)
\leq
C_{p}E\Big[\int_0^{T\wedge\tau_{J}}\|V^{\varepsilon}(s)-V^{0}(s)\|_{\mathbb{V}}^{p-2}
      \|u^\varepsilon(s)-u^0(s)\|_{\mathbb{V}}^2ds\Big]\nonumber\\
&\leq&
C_{p}E\Big[\int_0^{T\wedge\tau_{J}}\|V^{\varepsilon}(s)-V^{0}(s)\|_{\mathbb{V}}^{p}ds\Big]
      +C_{p}E\Big[\int_0^{T\wedge\tau_{J}}\|u^{\varepsilon}(s)-u^{0}(s)\|_{\mathbb{V}}^{p}ds\Big]\nonumber\\
&\leq&
C_{p}\int_0^{T}E\Big[\sup_{r\in[0,s\wedge\tau_{J}]}\|V^{\varepsilon}(r)-V^{0}(r)\|_{\mathbb{V}}^{p}\Big]ds
      +\varepsilon^{\frac{p}{2}}C_{p,T}E\Big[\sup_{s\in[0,T]}\|V^{\varepsilon}(s)\|_{\mathbb{V}}^{p}\Big].
\end{eqnarray}
and
\begin{eqnarray}\label{CLT 05}
& &J_{5}(T)
\leq
C_{p}E\Big[\int_0^{T\wedge\tau_{J}}\|V^{\varepsilon}(s)-V^{0}(s)\|_{\mathbb{V}}^{p-2}\|u^\varepsilon(s)-u^0(s)\|_{\mathbb{V}}^2ds\Big]\nonumber\\
&\leq&
C_{p}\int_0^{T}E\Big[\sup_{r\in[0,s\wedge\tau_{J}]}\|V^{\varepsilon}(r)-V^{0}(r)\|_{\mathbb{V}}^{p}\Big]ds
      +\varepsilon^{\frac{p}{2}}C_{p,T}E\Big[\sup_{s\in[0,T]}\|V^{\varepsilon}(s)\|_{\mathbb{V}}^{p}\Big].
\end{eqnarray}
Combining (\ref{CLT 00})--(\ref{CLT 05}), we get
\begin{eqnarray}
& &(1-\delta)E\Big[\sup_{s\in[0,T\wedge\tau_{J}]}\|V^\epsilon(s)-V^{0}(s)\|^p_\mathbb{V}\Big]\nonumber\\
& &\quad+\nu pE\Big[\int_0^{T\wedge\tau_{J}}\|V^{\varepsilon}(s)-V^{0}(s)\|_{\mathbb{V}}^{p-2}\|V^{\varepsilon}(s)-V^{0}(s)\|^{2}ds\Big]\nonumber\\
&\leq&
C_{p,\delta}\int_0^{T}E\Big[\sup_{r\in[0,s\wedge\tau_{J}]}\|V^{\varepsilon}(r)-V^{0}(r)\|_{\mathbb{V}}^{p}\Big]\big(\|u^0(s)\|_{\mathbb{W}}+1\big)ds\nonumber\\
& &+\sqrt{\varepsilon}C_{p,\delta,T}
      \Big(E\Big[\sup_{s\in[0,T]}\|V^{\varepsilon}(s)\|_{\mathbb{V}}^{2p}\Big]
            +E\Big[\sup_{s\in[0,T]}\|V^0(s)\|_{\mathbb{W}}^p\Big]\Big)\nonumber\\
& &+\varepsilon^{\frac{p}{2}}C_{p}E\Big[\sup_{s\in[0,T\wedge\tau_{J}]}\|V^{\varepsilon}(s)\|_{\mathbb{V}}^{p}\Big].
\end{eqnarray}
Choosing $\delta=\frac{1}{2}$ and applying Gronwall's inequality, we obtain
\begin{eqnarray}
E\Big[\sup_{s\in[0,T\wedge\tau_{J}]}\|V^\epsilon(s)-V^{0}(s)\|^p_\mathbb{V}\Big]\leq(\varepsilon^{\frac{p}{2}}+\sqrt{\varepsilon})C_p.
\end{eqnarray}
Let $J\rightarrow\infty$ and $\varepsilon\rightarrow0$ to obtain (\ref{CLT limit}).
\end{proof}

\section{Moderate deviation principle}
In this section, we will prove that $\frac{1}{\sqrt{\varepsilon}\lambda(\varepsilon)}(u^{\varepsilon}-u^{0})$ satisfies an LDP on $C([0,T];\mathbb{V})$ with $\lambda(\varepsilon)$ satisfying (\ref{condition}). This special type of LDP is usually called the moderate deviation of $u^{\varepsilon}$.
\subsection{Weak convergence method}
We will recall the general criteria for a large deviation principle obtained in \cite{Budhiraja-Dupuis}.
Let $\mathcal{E}$ be a Polish space with the Borel $\sigma$-field $\mathcal{B}(\mathcal{E})$.
    \begin{dfn}\label{Dfn-Rate function}
       \emph{\textbf{(Rate function)}} A function $I: \mathcal{E}\rightarrow[0,\infty]$ is called a rate function on
       $\mathcal{E}$,
       if for each $M<\infty$, the level set $\{x\in\mathcal{E}:I(x)\leq M\}$ is a compact subset of $\mathcal{E}$.
         \end{dfn}
    \begin{dfn}
       \emph{\textbf{(Large deviation principle)}} Let $I$ be a rate function on $\mathcal{E}$.  A family
       $\{X^\e\}$ of $\EE$-valued random elements is  said to satisfy a large deviation principle on $\mathcal{E}$
       speed $\lambda^2(\varepsilon)$ and with rate function $I$, if the following two conditions
       hold:
       \begin{itemize}
         \item[$(a)$](Upper bound) For each closed subset $F$ of $\mathcal{E}$,
              $$
                \limsup_{\e\rightarrow 0}\frac{1}{\lambda^2(\varepsilon)}\log P(X^\e\in F)\leq- \inf_{x\in F}I(x).
              $$
         \item[$(b)$](Lower bound) For each open subset $G$ of $\mathcal{E}$,
              $$
                \liminf_{\e\rightarrow 0}\frac{1}{\lambda^2(\varepsilon)}\log P(X^\e\in G)\geq- \inf_{x\in G}I(x).
              $$
       \end{itemize}
    \end{dfn}

\vskip0.3cm

The Cameron-Martin space associated with the Wiener process $\{W(t), t\in[0,T]\}$ is given by
\beq\label{Cameron-Martin}
\HH_0:=\left\{h:[0,T]\rightarrow \mathbb{R}^m; h \  \text{is absolutely continuous and } \int_0^T\|\dot h(s)\|_{\mathbb{R}^m}^2ds<+\infty\right\}.
\nneq
The space $\HH_0$ is a Hilbert space with inner product
 $$
 \langle h_1, h_2\rangle_{\HH_0}:=\int_0^T\langle \dot h_1(s), \dot h_2(s)\rangle_{\mathbb{R}^m}ds.
 $$

Let $\AA$ denote  the class of $\mathbb{R}^m$-valued $\{\FF_t\}$-predictable processes $\phi$ belonging to $\HH_0$ P-a.s..
Set $S_N=\{h\in \HH_0; \int_0^T\|\dot h(s)\|_{\mathbb{R}^m}^2ds\le N\}$. The set $S_N$ endowed with the weak topology is a Polish space.
Define $\AA_N=\{\phi\in \AA;\phi(\omega)\in S_N, P\text{-a.s.}\}$.

\vskip0.3cm

 Recall the following result from \cite{Budhiraja-Dupuis}.

\bthm\label{thm BD}{ For $\e>0$, let $\Gamma^\e$ be a measurable mapping from $C([0,T];\mathbb{R}^m)$ into $\EE$.
Let $X^\e:=\Gamma^\e(W(\cdot))$. Suppose that
there  exists a measurable map $\Gamma^0:C([0,T];\mathbb{R}^m)\rightarrow \EE$ such that
\begin{itemize}
   \item[(a)] for every $N<+\infty$ and any family $\{h^\e;\e>0\}\subset \AA_N$ satisfying that $h^\e$ converges in distribution as $S_N$-valued random elements to $h$ as $\e\rightarrow 0$,\\
    $\Gamma^\e\left(W(\cdot)+\lambda(\e)\int_0^{\cdot}\dot h^\e(s)ds\right)$ converges in distribution to $\Gamma^0(\int_0^{\cdot}\dot h(s)ds)$ as $\e\rightarrow 0$;
   \item[(b)] for every $N<+\infty$, the set
   $$
 \left\{\Gamma^0\left(\int_0^{\cdot}\dot h(s)ds\right); h\in S_N\right\}
  $$
   is a compact subset of $\EE$.
 \end{itemize}
Then the family $\{X^\e\}_{\e>0}$ satisfies a large deviation principle in $\EE$ with the rate function $I$ given by
\beq\label{rate function}
I(g):=\inf_{\{h\in \HH_0;g=\Gamma^0(\int_0^{\cdot}\dot h(s)ds)\}}\left\{\frac12\int_0^T\|\dot h(s)\|_{\mathbb{R}^m}^2ds\right\},\ g\in\EE,
\nneq
with the convention $\inf\{\emptyset\}=\infty$.
 }\nthm

\subsection{Our main result}

\vskip 0.2cm

Set $Z^{\varepsilon}(t):=(u^{\varepsilon}(t)-u^{0}(t))/\sqrt{\varepsilon}\lambda(\varepsilon)$. It is easy to see that $Z^{\e}$ satisfies the following stochastic evolution equation:
\begin{eqnarray}\label{MDP 01}
& &dZ^{\varepsilon}(t)+\nu\widehat{A}Z^{\varepsilon}(t)dt+\widehat{B}\big(Z^{\varepsilon}(t),u^{0}(t)+\sqrt{\varepsilon}\lambda(\varepsilon)Z^{\varepsilon}(t)\big)dt
+\widehat{B}\big(u^{0}(t),Z^{\varepsilon}(t)\big)dt\nonumber\\
&=&
\frac{1}{\sqrt{\e}\lambda(\varepsilon)}\big(\widehat{F}(u^0(t)+\sqrt{\varepsilon}\lambda(\varepsilon)Z^{\varepsilon}(t),t)-\widehat{F}(u^0(t),t)\big)dt\nonumber\\
& &+\frac{1}{\lambda(\e)}\widehat{G}\big(u^{0}(t)+\sqrt{\varepsilon}\lambda(\varepsilon)Z^{\varepsilon}(t),t\big)dW(t).
\end{eqnarray}

\vskip 0.2cm

The solution of (\ref{MDP 01}) determines a mapping $\Gamma^{\e}$ from $C(0,T;R^{m})$ to $C(0,T;\mathbb{V})$ so that $\Gamma^{\e}(W)=Z^{\e}$.
Let $N$ be any fixed positive integer. Fixed $h\in S_{N}$, consider the deterministic PDE:
\begin{eqnarray}\label{MDP 02}
dX(t)+\nu\widehat{A}X(t)dt+\widehat{B}(X(t),u^0(t))dt+\widehat{B}(u^0(t),X(t))dt\nonumber\\
\quad=\widehat{F}'(u^0(t),t)X(t)dt+\widehat{G}(u^0(t),t)\dot{h}(t)dt.
\end{eqnarray}

\vskip 0.2cm

Theorem \ref{Thm condition 02} below implies that there exists a unique solution $X(t)\in C(0,T;\mathbb{V})$ to the equation (\ref{MDP 02}). Define
$$ \Gamma^{0}(\int_{0}^{\cdot}\dot{h}(s)ds):=X. $$

\vskip 0.2cm

Let $I:C(0,T;\mathbb{V})\rightarrow[0,\infty]$ be defined as in (\ref{rate function}).
\begin{thm}\label{th main}
Assume that {\bf (F1)}, {\bf (F2)}, {\bf (G)} and {\bf (I)} hold. Then the family $\{Z^{\e}\}_{\e>0}$ of system (\ref{MDP 01}) satisfied a large deviation principle on $C([0,T];\mathbb{V})$ with speed $\lambda^2(\e)$ and with the good rate function I with respect to the topology of uniform convergence.
\end{thm}

\noindent{\bf Proof of Theorem \ref{th main}.}

According to Theorem \ref{thm BD}, we need to prove that Condition (a), (b) are fulfilled. The verification of Condition (a) will be given
 by Theorem \ref{Thm condition 02} below. Condition (b) will be established in Theorem \ref{Thm condition 01} below.

\subsection{The Proofs}

\vskip 0.2cm

For any fixed family $\{h^{\e};\e>0\}\subset \AA_N$, By Girsanov Transformation and the definition of $P^{\e}$, we know that  $X^{h^{\e}}:=\Gamma^{\e}(W(\cdot)+\lambda(\e)\int_{0}^{\cdot}\dot{h}^{\e}(s)ds)$ is the solution of the following SPDE:
\begin{eqnarray}\label{MDP 03}
& &dX^{h^{\varepsilon}}(t)+\nu\widehat{A}X^{h^{\varepsilon}}(t)dt+\widehat{B}\big(X^{h^{\varepsilon}}(t),u^{0}(t)
+\sqrt{\varepsilon}\lambda(\varepsilon)X^{h^{\varepsilon}}(t)\big)dt+\widehat{B}\big(u^{0}(t),X^{h^{\varepsilon}}(t)\big)dt\nonumber\\
&=&
\frac{1}{\sqrt{\e}\lambda(\varepsilon)}\big(\widehat{F}(u^0(t)+\sqrt{\varepsilon}\lambda(\varepsilon)X^{h^{\varepsilon}}(t),t)-\widehat{F}(u^0(t),t)\big)dt\nonumber\\
& &+\frac{1}{\lambda(\e)}\widehat{G}\big(u^{0}(t)+\sqrt{\varepsilon}\lambda(\varepsilon)X^{h^{\varepsilon}}(t),t\big)dW(t)
+\widehat{G}\big(u^{0}(t)+\sqrt{\varepsilon}\lambda(\varepsilon)X^{h^{\varepsilon}}(t),t\big)\dot{h}^{\e}(t)dt.
\end{eqnarray}
\vskip 0.2cm

First we will establish some priori estimates.
\begin{lem}\label{MDP lemma 01}
There exists $\e_{0}>0$ and a constant $C_{p,N}$ such that
\begin{eqnarray}\label{MDP Estimation 01}
\sup_{\e\in(0,\e_{0})}E\Big[\sup_{s\in[0,T]}\|X^{h^{\varepsilon}}(s)\|_{\mathbb{W}}^p\Big]\leq C_{p,N},\ \ \text{for any }2\leq p<\infty.
\end{eqnarray}
\end{lem}
Set $\mathbb{W}_M=Span(e_1,\cdots,e_M)$. Let $X^{M,h^{\e}}\in\mathbb{W}_M$ be the Galerkin approximations of (\ref{MDP 03}) satisfying
\begin{eqnarray}\label{MDP 04}
& &d(X^{M,h^{\e}},e_i)_{\mathbb{V}}+\nu((X^{M,h^{\e}},e_i))dt\nonumber\\
&=&-\langle\widehat{B}(X^{M,h^{\e}},u^{0}+\sqrt{\e}\lambda(\e)X^{M,h^{\e}}),e_i\rangle dt
-\langle\widehat{B}(u^{0},X^{M,h^{\e}}),e_i\rangle dt\nonumber\\
& &+\frac{1}{\sqrt{\e}\lambda(\varepsilon)}\big(F(u^{0}+\sqrt{\e}\lambda(\e)X^{M,h^{\e}},t)-F(u^{0},t),e_i\big)dt\nonumber\\
& &+\frac{1}{\lambda(\e)}\big(G(u^{0}+\sqrt{\e}\lambda(\e)X^{M,h^{\e}},t),e_i\big)dW(t)\nonumber\\
& &+(G(u^{0}+\sqrt{\e}\lambda(\e)X^{M,h^{\e}},t)\dot{h}^{\e}(t),e_i)dt,\ \i=1,2,\cdots,M.
\end{eqnarray}
\vskip 0.2cm

As in the proof of Theorem 3.4 in \cite{RS-12}, one can show that $\lim_{M\rightarrow\infty}X^{M,h^{\e}}=X^{h^{\e}}$ with respect to the weak-star topology in $L^p(\Omega,\mathcal{F},P,L^\infty([0,T],\mathbb{W}))$ for any $p\geq 2$. Hence Lemma \ref{MDP lemma 01} will follow from the following result, whose proof is rather involved.
\begin{lem}\label{MDP lemma 02}
For $p\geq2$, we have
\begin{eqnarray}\label{MDP Estimation 02}
\sup_{\e\in(0,\e_{0})}E\Big[\sup_{s\in[0,T]}\|X^{M,h^{\varepsilon}}(s)\|_{\mathbb{V}}^p\Big]\leq C_{p,N},
\end{eqnarray}
\begin{eqnarray}\label{MDP Estimation 03}
\sup_{\e\in(0,\e_{0})}E\Big[\sup_{s\in[0,T]}\|X^{M,h^{\varepsilon}}(s)\|_{\mathbb{W}}^p\Big]\leq C_{p,N}.
\end{eqnarray}
\end{lem}
\begin{proof}
Recall $\|v\|_*=|curl(v-\alpha\Delta v)|\ {\rm for\ any}\ v\in\mathbb{W}.$ Define
$$\tau_{J}=inf\{t\geq0,\|X^{M,h^{\e}}(t)\|_{\mathbb{V}}+\|X^{M,h^{\e}}(t)\|_{\ast}\geq J\}.$$
Applying ${\rm It\hat{o}}$'s formula,
\begin{eqnarray*}
& &d(X^{M,h^{\e}},e_i)_{\mathbb{V}}^2+2\big(X^{M,h^{\e}},e_i\big)_{\mathbb{V}}\Big(\nu((X^{M,h^{\e}},e_i))dt\nonumber\\
& &\quad+\langle\widehat{B}(X^{M,h^{\e}},u^{0}+\sqrt{\e}\lambda(\e)X^{M,h^{\e}}),e_i\rangle dt
+\langle\widehat{B}(u^{0},X^{M,h^{\e}}),e_i\rangle dt\Big)\nonumber\\
&=&
2\big(X^{M,h^{\e}},e_i\big)_{\mathbb{V}}\Big(\frac{1}{\sqrt{\e}\lambda(\varepsilon)}\big(F(u^{0}+
\sqrt{\e}\lambda(\e)X^{M,h^{\e}},t)-F(u^{0},t),e_i\big)dt\nonumber\\
& &+\frac{1}{\lambda(\e)}\big(G(u^{0}+\sqrt{\e}\lambda(\e)X^{M,h^{\e}},t),e_i\big)dW(t)
+\big(G(u^{0}+\sqrt{\e}\lambda(\e)X^{M,h^{\e}},t)\dot{h}^{\e}(t),e_i\big)dt\Big)\nonumber\\
& &+\frac{1}{\lambda(\e)^2}\big(G(u^{0}+\sqrt{\e}\lambda(\e)X^{M,h^{\e}},t),e_i\big)
\big(G(u^{0}+\sqrt{\e}\lambda(\e)X^{M,h^{\e}},t),e_i\big)'dt.
\end{eqnarray*}
Noting that $\|X^{M,h^{\e}}\|^2_\mathbb{V}=\sum_{i=1}^M\lambda_i(X^{M,h^{\e}},e_i)^2_\mathbb{V}$ and $\langle\widehat{B}(u,v),v\rangle=0$, it follows that
\begin{eqnarray}\label{MDP ito 1}
& &d\|X^{M,h^{\e}}\|_{\mathbb{V}}^2+2\nu\|X^{M,h^{\e}}\|^2dt
+2\langle\widehat{B}(X^{M,h^{\e}},u^{0}),X^{M,h^{\e}}\rangle dt\nonumber\\
&=&
\frac{2}{\sqrt{\e}\lambda(\varepsilon)}\big(F(u^{0}+
      \sqrt{\e}\lambda(\e)X^{M,h^{\e}},t)-F(u^{0},t),X^{M,h^{\e}}\big)dt\nonumber\\
& &+\frac{2}{\lambda(\e)}\big(G(u^{0}+\sqrt{\e}\lambda(\e)X^{M,h^{\e}},t),X^{M,h^{\e}}\big)dW(t)\nonumber\\
& &+2\big(G(u^{0}+\sqrt{\e}\lambda(\e)X^{M,h^{\e}},t)\dot{h}^{\e}(t),X^{M,h^{\e}}\big)dt\nonumber\\
& &+\frac{1}{\lambda^2(\e)}\sum_{i=1}^{M}\lambda_{i}\big(G(u^{0}+\sqrt{\e}\lambda(\e)X^{M,h^{\e}},t),e_i\big)
      \big(G(u^{0}+\sqrt{\e}\lambda(\e)X^{M,h^{\e}},t),e_i\big)'dt.\nonumber\\
\end{eqnarray}
By ${\rm It\hat{o}}$'s formula again, for $p\geq4$,
\begin{eqnarray}
& &d\|X^{M,h^{\e}}\|_{\mathbb{V}}^p\nonumber\\
&=&\frac{p}{2}\|X^{M,h^{\e}}\|_{\mathbb{V}}^{p-2}
      \Big(
            -2\nu\|X^{M,h^{\e}}\|^2dt-2\langle\widehat{B}(X^{M,h^{\e}},u^{0}),X^{M,h^{\e}}\rangle\nonumber\\
         & &\quad+\frac{2}{\sqrt{\e}\lambda(\varepsilon)}\big(F(u^{0}+
                  \sqrt{\e}\lambda(\e)X^{M,h^{\e}},t)-F(u^{0},t),X^{M,h^{\e}}\big)dt\nonumber\\
         & &\quad+\frac{2}{\lambda(\e)}\big(G(u^{0}+\sqrt{\e}\lambda(\e)X^{M,h^{\e}},t),X^{M,h^{\e}}\big)dW(t)\nonumber\\
         & &\quad+2\big(G(u^{0}+\sqrt{\e}\lambda(\e)X^{M,h^{\e}},t)\dot{h}^{\e}(t),X^{M,h^{\e}}\big)dt\\
         & &\quad+\frac{1}{\lambda^2(\e)}\sum_{i=1}^{M}\lambda_i \big(G(u^{0}+\sqrt{\e}\lambda(\e)X^{M,h^{\e}},t),e_i\big)
                  \big(G(u^{0}+\sqrt{\e}\lambda(\e)X^{M,h^{\e}},t),e_i\big)'dt
      \Big)\nonumber\\
& &+p(\frac{p}{2}-1)\frac{1}{\lambda^2(\e)}\|X^{M,h^{\e}}\|_{\mathbb{V}}^{p-4}\nonumber\\
& &\qquad\big(G(u^{0}+\sqrt{\e}\lambda(\e)X^{M,h^{\e}},t),X^{M,h^{\e}}\big)\big(G(u^{0}+\sqrt{\e}\lambda(\e)X^{M,h^{\e}},t),X^{M,h^{\e}}\big)'dt.\nonumber
\end{eqnarray}
Integrate from 0 to $t$, take sup over the interval $[0,T\wedge\tau_{J}]$ and take expectation to get
\begin{eqnarray}\label{MDP V 00}
& &E\Big[\sup_{t\in[0,T\wedge\tau_{J}]}\|X^{M,h^{\e}}\|_{\mathbb{V}}^p\Big]+\nu pE\Big[\int_0^{T\wedge\tau_{J}}\|X^{M,h^{\e}}\|_{\mathbb{V}}^{p-2}\|X^{M,h^{\e}}\|^2dt\Big]\nonumber\\
&\leq&
pE\Big[\int_0^{T\wedge\tau_{J}}\|X^{M,h^{\e}}\|_{\mathbb{V}}^{p-2}
      |\langle\widehat{B}(X^{M,h^{\e}},u^{0}),X^{M,h^{\e}}\rangle|dt\Big]\nonumber\\
& &+\frac{p}{\sqrt{\e}\lambda(\varepsilon)}E\Big[\int_0^{T\wedge\tau_{J}}\|X^{M,h^{\e}}\|_{\mathbb{V}}^{p-2}\big|\big(F(u^{0}+
      \sqrt{\e}\lambda(\e)X^{M,h^{\e}},t)-F(u^{0},t),X^{M,h^{\e}}\big)\big|dt\Big]\nonumber\\
& &+\frac{p}{\lambda(\e)}E\Big[\sup_{t\in[0,T\wedge\tau_{J}]}\int_0^{t}\|X^{M,h^{\e}}\|_{\mathbb{V}}^{p-2}
      \big(G(u^{0}+\sqrt{\e}\lambda(\e)X^{M,h^{\e}},s),X^{M,h^{\e}}\big)dW(s)\Big]\nonumber\\
& &+pE\Big[\int_0^{T\wedge\tau_{J}}\|X^{M,h^{\e}}\|_{\mathbb{V}}^{p-2}
      \big|\big(G(u^{0}+\sqrt{\e}\lambda(\e)X^{M,h^{\e}},t)\dot{h}^{\e}(t),X^{M,h^{\e}}\big)\big|dt\Big]\nonumber\\
& &+\frac{p}{\lambda^2(\e)}E\Big[\int_0^{T\wedge\tau_{J}}\|X^{M,h^{\e}}\|_{\mathbb{V}}^{p-2}
      \sum_{i=1}^{M}\lambda_{i}\big(G(u^{0}+\sqrt{\e}\lambda(\e)X^{M,h^{\e}},t),e_i\big)\nonumber\\
& &\qquad\big(G(u^{0}+\sqrt{\e}\lambda(\e)X^{M,h^{\e}},t),e_i\big)'dt\Big]\nonumber\\
& &+\frac{p}{\lambda^2(\e)}(\frac{p}{2}-1)E\Big[\int_0^{T\wedge\tau_{J}}\|X^{M,h^{\e}}\|_{\mathbb{V}}^{p-4}\nonumber\\
& &\qquad\big(G(u^{0}+\sqrt{\e}\lambda(\e)X^{M,h^{\e}},t),X^{M,h^{\e}}\big)\big(G(u^{0}+\sqrt{\e}\lambda(\e)X^{M,h^{\e}},t),X^{M,h^{\e}}\big)dt\Big]\nonumber\\
&:=&
I_{1}(T)+I_{2}(T)+I_{3}(T)+I_{4}(T)+I_{5}(T)+I_{6}(T).
\end{eqnarray}
By Lemma \ref{Lem-B-01}, we have
\begin{eqnarray}\label{MDP V 01}
I_{1}(T)
&\leq&
C_p\int_0^{T}E\Big[\sup_{s\in[0,t\wedge\tau_{J}]}\|X^{M,h^{\e}}(s)\|_{\mathbb{V}}^p\Big]\|u^{0}(t)\|_{\mathbb{W}}dt.
\end{eqnarray}

By (\ref{F-02}),
\begin{eqnarray}\label{MDP V 02}
I_{2}(T)
&\leq&
\frac{C_p}{\sqrt{\e}\lambda(\e)}E\Big[\int_0^{T\wedge\tau_{J}}\|X^{M,h^{\e}}(t)\|_{\mathbb{V}}^{p-2}\|{\sqrt{\e}\lambda(\e)}X^{M,h^{\e}}(t)\|_{\mathbb{V}}
      \|X^{M,h^{\e}}(t)\|_{\mathbb{V}}dt\Big]\nonumber\\
&\leq&
C_p\int_0^{T}E\Big[\sup_{s\in[0,t\wedge\tau_{J}]}\|X^{M,h^{\e}}(s)\|_{\mathbb{V}}^p\Big]dt.
\end{eqnarray}
Applying (\ref{G-02}) and B-D-G inequality to $I_3$, we get
\begin{eqnarray}\label{MDP V 03}
I_{3}(T)
&\leq&
\frac{C_p}{\lambda(\e)}E\Big[\int_0^{T\wedge\tau_{J}}\|X^{M,h^{\e}}(t)\|_{\mathbb{V}}^{2p-2}
       \|u^{0}(t)+\sqrt{\e}\lambda(\e)X^{M,h^{\e}}(t)\|_{\mathbb{V}}^{2}dt\Big]^{\frac{1}{2}}\nonumber\\
&\leq&
\frac{C_p}{\lambda(\e)}E\Big[\int_0^{T\wedge\tau_{J}}\|X^{M,h^{\e}}(t)\|_{\mathbb{V}}^{2p-2}\|u^{0}(t)\|_{\mathbb{V}}^{2}dt]^{\frac{1}{2}}\nonumber\\
& &\quad+C_p\sqrt{\varepsilon}E[\int_0^{T\wedge\tau_{J}}\|X^{M,h^{\e}}(t)\|_{\mathbb{V}}^{2p}dt\Big]^{\frac{1}{2}}\nonumber\\
&\leq&
\frac{C_p}{\lambda(\e)}E\Big[\sup_{t\in[0,T\wedge\tau_{J}]}\|X^{M,h^{\e}}(t)\|_{\mathbb{V}}^{p-1}
      \big(\int_0^{T\wedge\tau_{J}}\|u^{0}(t)\|_{\mathbb{V}}^{2}dt\big)^{\frac{1}{2}}\Big]\nonumber\\
& &\quad+C_p\sqrt{\e}E\Big[\sup_{t\in[0,T\wedge\tau_{J}]}\|X^{M,h^{\e}}(t)\|_{\mathbb{V}}^{\frac{p}{2}}
      \big(\int_0^{T\wedge\tau_{J}}\|X^{M,h^{\e}}(t)\|_{\mathbb{V}}^{p}dt\big)^{\frac{1}{2}}\Big]\nonumber\\
&\leq&
\eta_{1}E\Big[\sup_{t\in[0,T\wedge\tau_{J}]}\|X^{M,h^{\e}}(t)\|_{\mathbb{V}}^{p}\Big]
      +\frac{C_{p,\eta_1}}{\lambda(\e)}E\Big[\int_0^{T\wedge\tau_{J}}\|u^{0}(t)\|_{\mathbb{V}}^{2}dt\Big]^{\frac{p}{2}}\nonumber\\
& &\quad+\eta_{2}E\Big[\sup_{t\in[0,T\wedge\tau_{J}]}\|X^{M,h^{\e}}(t)\|_{\mathbb{V}}^{p}\Big]
      +C_{p,\eta_2}\varepsilon E\Big[\int_0^{T\wedge\tau_{J}}\|X^{M,h^{\e}}(t)\|_{\mathbb{V}}^{p}dt\Big]\nonumber\\
&\leq&
(\eta_{1}+\eta_{2})E\Big[\sup_{t\in[0,T\wedge\tau_{J}]}\|X^{M,h^{\e}}(t)\|_{\mathbb{V}}^{p}\Big]
+C_{p,\eta_2}\varepsilon\int_0^{T}E\Big[\sup_{s\in[0,t\wedge\tau_{J}]}\|X^{M,h^{\e}}(s)\|_{\mathbb{V}}^p\Big]dt\nonumber\\
& &\quad+\frac{C_{p,\eta_1,T}}{\lambda^p(\varepsilon)}\sup_{t\in[0,T]}\|u^{0}(t)\|_{\mathbb{V}}^{p},
\end{eqnarray}
where $\eta_1,\eta_2$ are positive constants chosen later. $I_4$ and $I_6$ can be bounded as follows:
\begin{eqnarray}\label{MDP V 04}
I_{4}(T)
&\leq&
C_pE\Big[\int_0^{T\wedge\tau_{J}}\|X^{M,h^{\e}}(t)\|_{\mathbb{V}}^{p-1}\|u^{0}(t)+\sqrt{\e}\lambda(\e)X^{M,h^{\e}}(t)\|_{\mathbb{V}}
      \|\dot{h}^{\e}(t)\|_{{\mathbb{R}}^m}dt\Big]\nonumber\\
&\leq&
C_pE\Big[\sup_{t\in[0,T\wedge\tau_{J}]}\|X^{M,h^{\e}}(t)\|_{\mathbb{V}}^{p-1}
      \int_0^{T\wedge\tau_{J}}\|u^{0}(t)+\sqrt{\e}\lambda(\e)X^{M,h^{\e}}(t)\|_{\mathbb{V}}\|\dot{h}^{\e}(t)\|_{{\mathbb{R}}^m}dt\Big]\nonumber\\
&\leq&
\delta E\Big[\sup_{t\in[0,T\wedge\tau_{J}]}\|X^{M,h^{\e}}(t)\|_{\mathbb{V}}^{p}\Big]
+C_{p,\delta}E\Big[\int_0^{T\wedge\tau_{J}}\|u^{0}(t)+\sqrt{\e}\lambda(\e)X^{M,h^{\e}}(t)\|_{\mathbb{V}}\|\dot{h}^{\e}(t)\|_{{\mathbb{R}}^m}dt\Big]^p\nonumber\\
&\leq&
\delta E\Big[\sup_{t\in[0,T\wedge\tau_{J}]}\|X^{M,h^{\e}}(t)\|_{\mathbb{V}}^{p}\Big]
+C_{p,\delta}N^{p/2}E\Big[\int_0^{T\wedge\tau_{J}}\|u^{0}(t)+\sqrt{\e}\lambda(\e)X^{M,h^{\e}}(t)\|_{\mathbb{V}}^2dt\Big]^{p/2}\nonumber\\
&\leq&
\delta E\Big[\sup_{t\in[0,T\wedge\tau_{J}]}\|X^{M,h^{\e}}(t)\|_{\mathbb{V}}^{p}\Big]
+C_{p,\delta,N}\varepsilon^{\frac{p}{2}}\lambda^p(\varepsilon)\int_0^{T}E\Big[\sup_{s\in[0,t\wedge\tau_{J}]}\|X^{M,h^{\e}}(s)\|_{\mathbb{V}}^p\Big]dt\nonumber\\
& &\quad+C_{p,\delta,N,T}\sup_{t\in[0,T]}\|u^{0}(t)\|_{\mathbb{V}}^{p},
\end{eqnarray}
where $\delta$ is a positive constant.
And
\begin{eqnarray}\label{MDP V 05}
I_{6}(T)
&\leq&
\frac{C_p}{\lambda^2(\e)}E\Big[\int_0^{T\wedge\tau_{J}}\|X^{M,h^{\e}}(t)\|_{\mathbb{V}}^{p-2}
      \|u^{0}(t)+\sqrt{\e}\lambda(\e)X^{M,h^{\e}}(t)\|_{\mathbb{V}}^{2}dt\Big]\nonumber\\
&\leq&
C_p\varepsilon E\Big[\int_0^{T\wedge\tau_{J}}\|X^{M,h^{\e}}(t)\|_{\mathbb{V}}^{p}dt\Big]
      +\frac{C_p}{\lambda^2(\e)}E\Big[\int_0^{T\wedge\tau_{J}}\|X^{M,h^{\e}}(t)\|_{\mathbb{V}}^{p-2}\|u^{0}(t)\|_{\mathbb{V}}^{2}dt\Big]\nonumber\\
&\leq&
C_p(\varepsilon+1)\int_0^{T}E\Big[\sup_{s\in[0,t\wedge\tau_{J}]}\|X^{M,h^{\e}}(s)\|_{\mathbb{V}}^p\Big]dt
      +\frac{C_{p,T}}{\lambda^p(\e)}\sup_{t\in[0,T]}\|u^{0}(t)\|_{\mathbb{V}}^{p}.
\end{eqnarray}
Now we estimate $I_5(T)$. By Lemma \ref{Lem GS}, there exists a unique solution $\widetilde{G}^{\e}(t)\in{\mathbb{W}}^{\otimes m}$ of the following equation:
\begin{eqnarray*}
\widetilde{G}^{\e}(t)-\alpha \Delta \widetilde{G}^{\e}(t)=G(u^{0}(t)+\sqrt{\e}\lambda(\e)X^{M,h^{\e}}(t),t)\ {\rm in}\ \mathcal{O},\nonumber\\
{\rm div}\ \widetilde{G}^{\e}(t)=0\ {\rm in}\ \mathcal{O},\\
\widetilde{G}^{\e}(t)=0\ {\rm on}\ \partial \mathcal{O}.\nonumber
\end{eqnarray*}
Moreover,
$$
(\widetilde{G}^{\e}(t),e_i)_\mathbb{V}=(G(u^{0}(t)+\sqrt{\e}\lambda(\e)X^{M,h^{\e}}(t),t),e_i),\ \ \forall i\in\{1,2,\cdots,M\},
$$
and there exists a positive constant $C_0$ such that
$$
\|\widetilde{G}^{\e}(t)\|_{\mathbb{W}^{\otimes m}}
\leq
C_0\|G(u^{0}(t)+\sqrt{\e}\lambda(\e)X^{M,h^{\e}}(t),t)\|_{\mathbb{V}^{\otimes m}}.
$$
Hence by (\ref{Basis}), (\ref{G-01}), (\ref{G-02}) and  $0<\lambda_i\leq\lambda_{i+1},\ i=1,2,\cdots$,
\begin{eqnarray}\label{MDP V 06*}
& &\sum_{i=1}^M\lambda_i(G(u^{0}(s)+\sqrt{\e}\lambda(\e)X^{M,h^{\e}}(s),s),e_i)(G(u^{0}(s)+\sqrt{\e}\lambda(\e)X^{M,h^{\e}}(s),s),e_i)'\nonumber\\
&=&
\sum_{i=1}^M\lambda_i(\widetilde{G}^{\e}(s),e_i)_\mathbb{V}
      (\widetilde{G}^{\e}(s),e_i)_\mathbb{V}'\nonumber\\
&=&
\sum_{i=1}^M\frac{1}{\lambda_i}(\widetilde{G}^{\e}(s),s),e_i)_\mathbb{W}
      (\widetilde{G}^{\e}(s),s),e_i)_\mathbb{W}'\nonumber\\
&\leq&
\frac{1}{\lambda_1}\|\widetilde{G}^{\e}(s),s)\|^2_{\mathbb{W}^{\otimes m}}\nonumber\\
&\leq&
\frac{C^2_0}{\lambda_1}\|G(u^{0}(s)+\sqrt{\e}\lambda(\e)X^{M,h^{\e}}(s),s)\|^2_{\mathbb{V}^{\otimes m}}\nonumber\\
&\leq&
C\|u^{0}(s)+\sqrt{\e}\lambda(\e)X^{M,h^{\e}}(s)\|^2_\mathbb{V}.
\end{eqnarray}
Hence
\begin{eqnarray}\label{MDP V 06}
I_{5}(T)
&\leq&
\frac{C_p}{\lambda^2(\e)}E\Big[\int_0^{T\wedge\tau_{J}}\|X^{M,h^{\e}}\|_{\mathbb{V}}^{p-2}\|u^{0}(t)+\sqrt{\e}\lambda(\e)X^{M,h^{\e}}(t)\|^2_\mathbb{V}dt\Big]\nonumber\\
&\leq&
C_p(\varepsilon+1)\int_0^{T}E\Big[\sup_{s\in[0,t\wedge\tau_{J}]}\|X^{M,h^{\e}}(s)\|_{\mathbb{V}}^p\Big]dt
      +\frac{C_{p,T}}{\lambda^p(\e)}\sup_{t\in[0,T]}\|u^{0}(t)\|_{\mathbb{V}}^{p}.
\end{eqnarray}
Combining (\ref{MDP V 00})--(\ref{MDP V 05}) and (\ref{MDP V 06}), we get
\begin{eqnarray}\label{MDP V 08}
& &(1-\delta-\eta_{1}-\eta_{2})E\Big[\sup_{t\in[0,T\wedge\tau_{J}]}\|X^{M,h^{\e}}(t)\|_{\mathbb{V}}^{p}\Big]
+\nu pE\Big[\int_0^{T\wedge\tau_{J}}\|X^{M,h^{\e}}\|_{\mathbb{V}}^{p-2}\|X^{M,h^{\e}}\|^2dt\Big]\nonumber\\
&\leq&
C_{p,\delta,\eta_1,\eta_2,T}\int_0^{T}E\Big[\sup_{s\in[0,t\wedge\tau_{J}]}\|X^{M,h^{\e}}(s)\|_{\mathbb{V}}^p\Big]\big(\|u^{0}(t)\|_{\mathbb{W}}+1\big)dt\nonumber\\
& &+C_{p,N,\delta,\eta_1,\eta_2,T}\sup_{t\in[0,T]}\|u^{0}(t)\|_{\mathbb{V}}^{p},\forall \e\in(0,1).
\end{eqnarray}
Choose $\delta=\eta_{1}=\eta_{2}=\frac{1}{4}.$ By Gronwall's Inequality, there exists a $\varepsilon_0>0$, such that
$$\sup_{\varepsilon\in(0,\varepsilon_0)}E[\sup_{t\in[0,T\wedge\tau_{J}]}\|X^{M,h^{\e}}(t)\|_{\mathbb{V}}^{p}]\leq C_{p,N},$$
which is (\ref{MDP Estimation 02}).
\vskip 0.2cm
Next we prove (\ref{MDP Estimation 03}). To this end, we need to establish an estimate for $\|X^{M,h^{\e}}(t)\|_{\ast}^{p}$.\\
Setting
\begin{eqnarray*}
& &\phi(u^{0},X^{M,h^{\e}})\nonumber\\
&=&-\nu\Delta X^{M,h^{\e}}+curl(X^{M,h^{\e}}-\alpha\Delta X^{M,h^{\e}})\times(u^{0}+\sqrt{\e}\lambda(\e)X^{M,h^{\e}})\nonumber\\
& &+curl(u^{0}-\alpha\Delta u^{0})\times X^{M,h^{\e}}
-\frac{1}{\sqrt{\e}\lambda(\e)}\big(F(u^{0}+\sqrt{\e}\lambda(\e)X^{M,h^{\e}},t)-F(u^{0},t)\big)\nonumber\\
& &-G(u^{0}+\sqrt{\e}\lambda(\e)X^{M,h^{\e}},t)\dot{h}^{\e}(t),
\end{eqnarray*}
(\ref{MDP 04}) becomes
\begin{eqnarray*}
d(X^{M,h^{\e}},e_i)_{\mathbb{V}}+(\phi(u^{M,0},X^{M,h^{\e}}),e_i)dt
=\frac{1}{\lambda(\e)}\big(G(u^{0}+\sqrt{\e}\lambda(\e)X^{M,h^{\e}},t),e_i\big)dW(t).
\end{eqnarray*}
Since $e_i\in\mathbb{H}^4(\mathcal{O}),\ i\in\mathbb{N}$ (see Lemma \ref{eq Basis}), $X^{M,h^{\e}}\in\mathbb{W}_M$, and $u^0 \in \mathbb{H}^4(\mathcal{O})$ (see Lemma \ref{Regularity}), we see that
 $\phi(u^{0},X^{M,h^{\e}})\in\mathbb{H}^1(\mathcal{O})$. By Lemma \ref{Lem GS}, there exists a unique $v^M\in\mathbb{W}$ satisfying
\begin{eqnarray*}
v^M-\alpha \Delta v^M=\phi(u^{0},X^{M,h^{\e}})\ {\rm in}\ \mathcal{O},\nonumber\\
{\rm div}\ v^M=0\ {\rm in}\ \mathcal{O},\\
v^M=0\ {\rm on}\ \partial \mathcal{O}.\nonumber
\end{eqnarray*}
Moreover,
$$
(v^M,e_i)_\mathbb{V}=(\phi(u^{0},X^{M,h^{\e}}),e_i),\ \ \forall i\in\{1,2,\cdots,M\}.
$$
Thus
\begin{eqnarray}\label{eq 12}
d(X^{M,h^{\e}},e_i)_\mathbb{V}+(v^M,e_i)_\mathbb{V}dt=\frac{1}{\lambda(\e)}\big(G(u^{0}+\sqrt{\e}\lambda(\e)X^{M,h^{\e}},t),e_i\big)dW(t).
\end{eqnarray}

Let $\widetilde{G}$ be defined as before so that
$$
\lambda_i(G(u^{0}+\sqrt{\e}\lambda(\e)X^{M,h^{\e}},t),e_i)=(\widetilde{G}^{\e}(t),e_i)_{\mathbb{W}}.
$$
Multiplying (\ref{eq 12}) by $\lambda_i$ and noticing (\ref{Basis}), we get
$$
d(X^{M,h^{\e}},e_i)_\mathbb{W}+(v^M,e_i)_\mathbb{W}dt=\frac{1}{\lambda(\e)}(\widetilde{G}^{\e}(t),e_i)_{\mathbb{W}}dW(t).
$$

Applying ${\rm It\hat{o}}$'s formula,
\begin{eqnarray*}
& &d(X^{M,h^{\e}},e_i)_\mathbb{W}^2+2(X^{M,h^{\e}},e_i)_\mathbb{W}(v^M,e_i)_\mathbb{W}dt\nonumber\\
&=&
\frac{2}{\lambda(\e)}(X^{M,h^{\e}},e_i)_\mathbb{W}(\widetilde{G}^{\e}(t),e_i)_{\mathbb{W}}dW(t)
+\frac{1}{\lambda^2(\e)}(\widetilde{G}^{\e}(t),e_i)_{\mathbb{W}}(\widetilde{G}^{\e}(t),e_i)'_{\mathbb{W}}dt.
\end{eqnarray*}
Taking summation from $i=1$ to $i=M$, we obtain
\begin{eqnarray*}
& &d\|X^{M,h^{\e}}(t)\|_\mathbb{W}^2+2(X^{M,h^{\e}},v^M)_\mathbb{W}dt\nonumber\\
&=&
\frac{2}{\lambda(\e)}\big(X^{M,h^{\e}},\widetilde{G}^{\e}(t)\big)_\mathbb{W}dW(t)
+\frac{1}{\lambda^2(\e)}\sum_{i=1}^{M}\big(\widetilde{G}^{\e}(t),e_i\big)_{\mathbb{W}}\big(\widetilde{G}^{\e}(t),e_i\big)'_{\mathbb{W}}dt.
\end{eqnarray*}
In view of (\ref{W}) and (\ref{Basis}), we rewrite the above equation as follows
\begin{eqnarray*}
& &d\big[\|X^{M,h^{\e}}\|_\mathbb{V}^2+\|X^{M,h^{\e}}\|_{\ast}^2\big]+
2\big[(X^{M,h^{\e}},v^M)_\mathbb{V}+\big(curl(X^{M,h^{\e}}-\alpha\Delta X^{M,h^{\e}}),curl(v^M-\alpha\Delta v^M)\big)\big]dt\nonumber\\
&=&
\frac{1}{\lambda^2(\e)}\sum_{i=1}^{M}\lambda_{i}^2\big(\widetilde{G}^{\e}(t),e_i\big)_{\mathbb{V}}
      \big(\widetilde{G}^{\e}(t),e_i\big)'_{\mathbb{V}}dt
+\frac{2}{\lambda(\e)}\big(X^{M,h^{\e}},\widetilde{G}^{\e}(t)\big)_\mathbb{V}dW(t)\nonumber\\
& &+\frac{2}{\lambda(\e)}\Big(curl(X^{M,h^{\e}}-\alpha\Delta X^{M,h^{\e}}),curl\big(\widetilde{G}^{\e}(t)-\alpha\Delta\widetilde{G}^{\e}(t)\big)\Big)dW(t).
\end{eqnarray*}
By definition of $v^M$ and $\widetilde{G}^{\e}(t)$, it follows that
\begin{eqnarray*}
& &d\big[\|X^{M,h^{\e}}\|_\mathbb{V}^2+\|X^{M,h^{\e}}\|_{\ast}^2\big]+2\big[(X^{M,h^{\e}},\phi(u^{0},X^{M,h^{\e}}))\nonumber\\
& &\qquad+\big(curl(X^{M,h^{\e}}-\alpha\Delta X^{M,h^{\e}}),curl(\phi(u^{0},X^{M,h^{\e}}))\big)\big]dt\nonumber\\
&=&
\frac{1}{\lambda^2(\e)}\sum_{i=1}^{M}\lambda_{i}^2\big(G(u^{0}+\sqrt{\e}\lambda(\e)X^{M,h^{\e}},t),e_i\big)
      \big(G(u^{0}+\sqrt{\e}\lambda(\e)X^{M,h^{\e}},t),e_i\big)'dt\nonumber\\
& &+\frac{2}{\lambda(\e)}\big(X^{M,h^{\e}},G(u^{0}+\sqrt{\e}\lambda(\e)X^{M,h^{\e}},t)\big)dW(t)\nonumber\\
& &+\frac{2}{\lambda(\e)}\Big(curl(X^{M,h^{\e}}-\alpha\Delta X^{M,h^{\e}}),
curl\big(G(u^{0}+\sqrt{\e}\lambda(\e)X^{M,h^{\e}},t)\big)\Big)dW(t).
\end{eqnarray*}
Subtracting (\ref{MDP ito 1}) from the above equation, we obtain
\begin{eqnarray}\label{MDP ito 2}
& &d\|X^{M,h^{\e}}\|_{\ast}^2+2\Big(curl(X^{M,h^{\e}}-\alpha\Delta X^{M,h^{\e}}),curl\big(\phi(u^{0},X^{M,h^{\e}})\big)\Big)dt\nonumber\\
&=&
\frac{1}{\lambda^2(\e)}\sum_{i=1}^{M}(\lambda_{i}^2-\lambda_{i})(G(u^{0}+\sqrt{\e}\lambda(\e)X^{M,h^{\e}},t),e_i)
(G(u^{M,0}+\sqrt{\e}\lambda(\e)X^{M,h^{\e}},t),e_i)'dt\nonumber\\
& &+\frac{2}{\lambda(\e)}\Big(curl(X^{M,h^{\e}}-\alpha\Delta X^{M,h^{\e}}),curl\big(G(u^{0}+\sqrt{\e}\lambda(\e)X^{M,h^{\e}},t)\big)\Big)dW(t).
\end{eqnarray}
A simple calculation gives that
$$\Big(curl(X^{M,h^{\e}}-\alpha\Delta X^{M,h^{\e}}),curl\big(curl(X^{M,h^{\e}}-\alpha\Delta X^{M,h^{\e}})\times(u^{0}+\sqrt{\e}\lambda(\e)X^{M,h^{\e}})\big)\Big)=0,$$
and
$$curl\big(curl(u^{0}-\alpha\Delta u^{0})\times X^{M,h^{\e}}\big)=-\Delta(u^{0}-\alpha\Delta u^{0})\times X^{M,h^{\e}}.$$
Hence,
\begin{eqnarray}\label{MDP curl 01}
& &\Big(curl(X^{M,h^{\e}}-\alpha\Delta X^{M,h^{\e}}),curl\big(\phi(u^{0},X^{M,h^{\e}})\big)\Big)\nonumber\\
&=&
\Big(curl\big(X^{M,h^{\e}}-\alpha\Delta X^{M,h^{\e}}\big),curl\big(-\nu\Delta X^{M,h^{\e}}\big)\Big)\nonumber\\
& &+\Big(curl\big(X^{M,h^{\e}}-\alpha\Delta X^{M,h^{\e}}\big),curl\big(curl(u^{0}-\alpha\Delta u^{0})\times X^{M,h^{\e}}\big)\Big)\nonumber\\
& &-\Big(curl(X^{M,h^{\e}}-\alpha\Delta X^{M,h^{\e}}),
      curl\big(\frac{1}{\sqrt{\e}\lambda(\e)}\big[F(u^{0}+\sqrt{\e}\lambda(\e)X^{M,h^{\e}},t)-F(u^{0},t)\big]\big)\Big)\nonumber\\
& &-\Big(curl(X^{M,h^{\e}}-\alpha\Delta X^{M,h^{\e}}),curl\big(G(u^{0}+\sqrt{\e}\lambda(\e)X^{M,h^{\e}},t)\dot{h}^{\e}(t)\big)\Big)\nonumber\\
&=&
\frac{\nu}{\alpha}\|X^{M,h^{\e}}\|_{\ast}^2-\frac{\nu}{\alpha}\Big(curl\big(X^{M,h^{\e}}-\alpha\Delta X^{M,h^{\e}}\big),curl\big(X^{M,h^{\e}}\big)\Big)\nonumber\\
& &+\Big(curl\big(X^{M,h^{\e}}-\alpha\Delta X^{M,h^{\e}}\big),-\Delta\big(u^{0}-\alpha\Delta u^{0}\big)\times X^{M,h^{\e}}\Big)\nonumber\\
& &-\Big(curl\big(X^{M,h^{\e}}-\alpha\Delta X^{M,h^{\e}}\big),curl\big(\frac{1}{\sqrt{\e}\lambda(\e)}\big[F(u^{0}+\sqrt{\e}\lambda(\e)X^{M,h^{\e}},t)-F(u^{0},t)\big]\big)\Big)\nonumber\\
& &-\Big(curl(X^{M,h^{\e}}-\alpha\Delta X^{M,h^{\e}}),curl\big(G(u^{0}+\sqrt{\e}\lambda(\e)X^{M,h^{\e}},t)\dot{h}^{\e}(t)\big)\Big).
\end{eqnarray}
Combining (\ref{MDP ito 2}) with (\ref{MDP curl 01}) yields that
\begin{eqnarray*}
& &d\|X^{M,h^{\e}}\|_{\ast}^2+\frac{2\nu}{\alpha}\|X^{M,h^{\e}}\|_{\ast}^2dt-\frac{2\nu}{\alpha}\Big(curl\big(X^{M,h^{\e}}-\alpha\Delta X^{M,h^{\e}}\big),curl\big(X^{M,h^{\e}}\big)\Big)dt\nonumber\\
& &+2\Big(curl\big(X^{M,h^{\e}}-\alpha\Delta X^{M,h^{\e}}\big),-\Delta\big(u^{0}-\alpha\Delta u^{0}\big)\times X^{M,h^{\e}}\Big)dt\nonumber\\
& &-2\Big(curl\big(X^{M,h^{\e}}-\alpha\Delta X^{M,h^{\e}}\big),curl\big(\frac{1}{\sqrt{\e}\lambda(\e)}\big[F(u^{0}+\sqrt{\e}\lambda(\e)X^{M,h^{\e}},t)-F(u^{0},t)\big]\big)\Big)dt\nonumber\\
& &-2\Big(curl\big(X^{M,h^{\e}}-\alpha\Delta X^{M,h^{\e}}\big),curl\big(G(u^{M,0}+\sqrt{\e}\lambda(\e)X^{M,h^{\e}},t)\dot{h}^{\e}(t)\big)\Big)dt\nonumber\\
&=&
\frac{1}{\lambda^2(\e)}\sum_{i=1}^{M}(\lambda_{i}^2-\lambda_{i})(G(u^{0}+\sqrt{\e}\lambda(\e)X^{M,h^{\e}},t),e_i)
(G(u^{0}+\sqrt{\e}\lambda(\e)X^{M,h^{\e}},t),e_i)'dt\nonumber\\
& &+\frac{2}{\lambda(\e)}\Big(curl\big(X^{M,h^{\e}}-\alpha\Delta X^{M,h^{\e}}\big),curl\big(G(u^{0}+\sqrt{\e}\lambda(\e)X^{M,h^{\e}},t)\big)\Big)dW(t).
\end{eqnarray*}
Applying ${\rm It\hat{o}}$'s formula to $\|X^{M,h^{\e}}\|_{\ast}^p$, $p\geq4$,
\begin{eqnarray*}
& &d\|X^{M,h^{\e}}\|_{\ast}^p\nonumber\\
&=&
p\|X^{M,h^{\e}}\|_{\ast}^{p-2}\Big\{-\frac{\nu}{\alpha}\|X^{M,h^{\e}}\|_{\ast}^2dt
+\frac{\nu}{\alpha}\Big(curl\big(X^{M,h^{\e}}-\alpha\Delta X^{M,h^{\e}}\big),curl\big(X^{M,h^{\e}}\big)\Big)dt\nonumber\\
& &+\Big(curl\big(X^{M,h^{\e}}-\alpha\Delta X^{M,h^{\e}}\big),\Delta\big(u^{0}-\alpha\Delta u^{0}\big)\times X^{M,h^{\e}}\Big)dt\nonumber\\
& &+\Big(curl\big(X^{M,h^{\e}}-\alpha\Delta X^{M,h^{\e}}\big),curl\big(\frac{1}{\sqrt{\e}\lambda(\e)}\big[F(u^{0}+\sqrt{\e}\lambda(\e)X^{M,h^{\e}},t)-F(u^{0},t)\big]\big)\Big)dt\nonumber\\
& &+\Big(curl\big(X^{M,h^{\e}}-\alpha\Delta X^{M,h^{\e}}\big),curl\big(G(u^{0}+\sqrt{\e}\lambda(\e)X^{M,h^{\e}},t)\dot{h}^{\e}(t)\big)\Big)dt\nonumber\\
& &+\frac{1}{2\lambda^2(\e)}\sum_{i=1}^{M}(\lambda_{i}^2-\lambda_{i})(G(u^{0}+\sqrt{\e}\lambda(\e)X^{M,h^{\e}},t),e_i)
(G(u^{0}+\sqrt{\e}\lambda(\e)X^{M,h^{\e}},t),e_i)'dt\nonumber\\
& &+\frac{1}{\lambda(\e)}\Big(curl\big(X^{M,h^{\e}}-\alpha\Delta X^{M,h^{\e}}\big),curl\big(G(u^{0}+\sqrt{\e}\lambda(\e)X^{M,h^{\e}},t)\big)\Big)dW(t)\Big\}\nonumber\\
& &+\frac{p}{\lambda^{2}(\e)}(\frac{p}{2}-1)\|X^{M,h^{\e}}\|_{\ast}^{p-4}\Big(curl\big(X^{M,h^{\e}}-\alpha\Delta X^{M,h^{\e}}\big),curl\big(G(u^{0}+\sqrt{\e}\lambda(\e)X^{M,h^{\e}},t)\big)\Big)\nonumber\\
& &\quad\cdot\Big(curl\big(X^{M,h^{\e}}-\alpha\Delta X^{M,h^{\e}}\big),curl\big(G(u^{0}+\sqrt{\e}\lambda(\e)X^{M,h^{\e}},t)\big)\Big)'dt.
\end{eqnarray*}
Integrate from $0$ to $t$, take sup over the interval $[0,T\wedge\tau_{J}]$ and take expectation to get
\begin{eqnarray}\label{MDP * 00}
& &E\Big[\sup_{t\in[0,T\wedge\tau_{J}]}\|X^{M,h^{\e}}(t)\|_{\ast}^p\Big]+\frac{\nu p}{\alpha}
E\Big[\int_0^{T\wedge\tau_{J}}\|X^{M,h^{\e}}(t)\|_{\ast}^{p}dt\Big]\nonumber\\
&\leq&
\frac{\nu p}{\alpha}E\Big[\int_0^{T\wedge\tau_{J}}\|X^{M,h^{\e}}(t)\|_{\ast}^{p-2}\Big|\Big(curl(X^{M,h^{\e}}(t)-\alpha\Delta X^{M,h^{\e}}(t)),curl(X^{M,h^{\e}}(t))\Big)\Big|dt\Big]\nonumber\\
& &+pE\Big[\int_0^{T\wedge\tau_{J}}\|X^{M,h^{\e}}(t)\|_{\ast}^{p-2}\nonumber\\
      & &\quad\cdot\Big|\Big(curl\big(X^{M,h^{\e}}(t)-\alpha\Delta X^{M,h^{\e}}(t)\big),\Delta\big(u^{0}(t)-\alpha\Delta u^{0}(t)\big)\times X^{M,h^{\e}}(t)\Big)\Big|dt\Big]\nonumber\\
& &+\frac{p}{\sqrt{\e}\lambda(\e)}E\Big[\int_0^{T\wedge\tau_{J}}\|X^{M,h^{\e}}(t)\|_{\ast}^{p-2}\Big|\Big(curl\big(X^{M,h^{\e}}(t)-\alpha\Delta X^{M,h^{\e}}(t)\big),\nonumber\\
      & &\qquad\qquad curl\big(\big[F(u^{0}(t)+\sqrt{\e}\lambda(\e)X^{M,h^{\e}}(t),t)-F(u^{0}(t),t)\big]\big)\Big)\Big|dt\Big]\nonumber\\
& &+pE\Big[\int_0^{T\wedge\tau_{J}}\|X^{M,h^{\e}}(t)\|_{\ast}^{p-2}\nonumber\\
      & &\quad\Big|\Big(curl\big(X^{M,h^{\e}}(t)-\alpha\Delta X^{M,h^{\e}}(t)\big),curl\big(G(u^{0}(t)+\sqrt{\e}\lambda(\e)X^{M,h^{\e}}(t),t)\dot{h}^{\e}(t)\big)\Big)\Big|dt]\nonumber\\
& &+\frac{p}{2\lambda^2(\e)}E\Big[\int_0^{T\wedge\tau_{J}}\|X^{M,h^{\e}}(t)\|_{\ast}^{p-2}\sum_{i=1}^{M}(\lambda_{i}^2-\lambda_{i})\nonumber\\
      & &\quad(G(u^{0}(t)+\sqrt{\e}\lambda(\e)X^{M,h^{\e}}(t),t),e_i)(G(u^{0}(t)+\sqrt{\e}\lambda(\e)X^{M,h^{\e}}(t),t),e_i)'dt\nonumber\\
& &+\frac{p}{\lambda(\e)}E\Big[\sup_{t\in[0,T\wedge\tau_{J}]}\Big|\int_0^{T\wedge\tau_{J}}\|X^{M,h^{\e}}(s)\|_{\ast}^{p-2}\nonumber\\
      & &\quad\Big(curl\big(X^{M,h^{\e}}(s)-\alpha\Delta X^{M,h^{\e}}(s)\big),curl\big(G(u^{0}(s)+\sqrt{\e}\lambda(\e)X^{M,h^{\e}}(s),s)\big)\Big)dW(s)\Big|\Big]\nonumber\\
& &+\frac{p}{\lambda^{2}(\e)}(\frac{p}{2}-1)E\Big[\int_0^{T\wedge\tau_{J}}\|X^{M,h^{\e}}(t)\|_{\ast}^{p-4}\nonumber\\
& &\quad\Big(curl\big(X^{M,h^{\e}}(t)-\alpha\Delta X^{M,h^{\e}}(t)\big),curl\big(G(u^{0}(t)+\sqrt{\e}\lambda(\e)X^{M,h^{\e}}(t),t)\big)\Big)\nonumber\\
& &\quad\Big(curl\big(X^{M,h^{\e}}(t)-\alpha\Delta X^{M,h^{\e}}(t)\big),curl\big(G(u^{0}(t)+\sqrt{\e}\lambda(\e)X^{M,h^{\e}}(t),t)\big)\Big)'dt\Big]\nonumber\\
&:=&J_{1}(T)+J_{2}(T)+J_{3}(T)+J_{4}(T)+J_{5}(T)+J_{6}(T)+J_{7}(T).
\end{eqnarray}
Applying Cauchy-Schwardz, Young inequalities, and noticing the fact
\begin{equation}\label{curl}
|curl(v)|^2\leq\frac{2}{\alpha}\|v\|_{\mathbb{V}}^2 \ \text{ for any } v\in\mathbb{W},
\end{equation}
we have
\begin{eqnarray}\label{MDP * 01}
J_{1}(T)
&\leq&
C_pE\Big[\int_0^{T\wedge\tau_{J}}\|X^{M,h^{\e}}(t)\|_{\ast}^{p-2}\|X^{M,h^{\e}}(t)\|_{\ast}\|X^{M,h^{\e}}(t)\|_{\mathbb{V}}dt\Big]\nonumber\\
&\leq&
C_p\int_0^{T}E\Big[\sup_{s\in[0,t\wedge\tau_{J}]}\|X^{M,h^{\e}}(s)\|_{\ast}^{p}\Big]dt+C_{p,T}E\Big[\sup_{t\in[0,T]}\|X^{M,h^{\e}}(t)\|_{\mathbb{V}}^{p}\Big].
\end{eqnarray}
By Lemma \ref{Regularity}, Cauchy-Schwardz, Young inequalities, we have
\begin{eqnarray}\label{MDP * 02}
J_{2}(T)
&\leq&
C_pE\Big[\int_0^{T\wedge\tau_{J}}\|X^{M,h^{\e}}(t)\|_{\ast}^{p-1}|\Delta(u^{0}(t)-\alpha\Delta u^{0}(t))\times X^{M,h^{\e}}(t)|dt\Big]\nonumber\\
&\leq&
C_pE\Big[\int_0^{T\wedge\tau_{J}}\|X^{M,h^{\e}}(t)\|_{\ast}^{p-1}|\Delta(u^{0}(t)-\alpha\Delta u^{0}(t))||X^{M,h^{\e}}(t)|_{L^{\infty}}dt\Big]\nonumber\\
&\leq&
C_pE\Big[\int_0^{T\wedge\tau_{J}}\|X^{M,h^{\e}}(t)\|_{\ast}^{p-1}\|u^{0}(t)\|_{\mathbb{H}^4}
(\|X^{M,h^{\e}}(t)\|_{\mathbb{V}}+\|X^{M,h^{\e}}(t)\|_{\ast})dt\Big]\nonumber\\
&\leq&
C_pE\Big[\int_0^{T\wedge\tau_{J}}\|X^{M,h^{\e}}(t)\|_{\ast}^{p}dt\Big]
+C_pE\Big[\sup_{t\in[0,T\wedge\tau_{J}]}\|X^{M,h^{\e}}(t)\|_{\ast}^{p-1}\int_0^{T\wedge\tau_{J}}\|X^{M,h^{\e}}(t)\|_{\mathbb{V}}dt\Big]\nonumber\\
&\leq&
C_p\int_0^{T}E\Big[\sup_{s\in[0,t\wedge\tau_{J}]}\|X^{M,h^{\e}}(s)\|_{\ast}^{p}\Big]dt
+\delta E\Big[\sup_{t\in[0,T\wedge\tau_{J}]}\|X^{M,h^{\e}}(t)\|_{\ast}^p\Big]\nonumber\\
& &\quad+C_{p,\delta,T}E\Big[\sup_{t\in[0,T]}\|X^{M,h^{\e}}(t)\|_{\mathbb{V}}^{p}\Big],
\end{eqnarray}
where $\delta$ is a positive constant.\\
By (\ref{F-01})--(\ref{G-02}) and (\ref{curl}), and using Cauchy-Schwardz and Young inequalities again,
\begin{eqnarray}\label{MDP * 03}
J_{3}(T)
&\leq&
C_pE\Big[\int_0^{T\wedge\tau_{J}}\|X^{M,h^{\e}}(t)\|_{\ast}^{p-1}\|X^{M,h^{\e}}(t)\|_{\mathbb{V}}dt\Big]\nonumber\\
&\leq&
C_p\int_0^{T}E[\sup_{s\in[0,t\wedge\tau_{J}]}\|X^{M,h^{\e}}(s)\|_{\ast}^{p}]dt+C_{p,T}E\Big[\sup_{t\in[0,T]}\|X^{M,h^{\e}}(t)\|_{\mathbb{V}}^{p}\Big],
\end{eqnarray}
and
\begin{eqnarray}\label{MDP * 04}
J_{4}(T)
&\leq&
C_pE\Big[\int_0^{T\wedge\tau_{J}}\|X^{M,h^{\e}}(t)\|_{\ast}^{p-1}
\|u^{0}(t)+\sqrt{\e}\lambda(\e)X^{M,h^{\e}}(t)\|_{\mathbb{V}}\|\dot{h^{\e}}\|_{\mathbb{R}^m}dt\Big]\nonumber\\
&\leq&
\eta E\Big[\sup_{t\in[0,T\wedge\tau_{J}]}\|X^{M,h^{\e}}(t)\|_{\ast}^p\Big]
+C_{p,\eta,N,T}E\Big[\sup_{t\in[0,T]}\|u^{0}(t)+\sqrt{\e}\lambda(\e)X^{M,h^{\e}}(t)\|_{\mathbb{V}}^{p}\Big]\nonumber\\
&\leq&
\eta E\Big[\sup_{t\in[0,T\wedge\tau_{J}]}\|X^{M,h^{\e}}(t)\|_{\ast}^p\Big]\nonumber\\
& &\quad+C_{p,\eta,N,T}E\Big[\sup_{t\in[0,T]}\big(\|u^{0}(t)\|_{\mathbb{V}}^{p}+
\varepsilon^{\frac{p}{2}}\lambda^p(\varepsilon)\|X^{M,h^{\e}}(t)\|_{\mathbb{V}}^{p}\big)\Big].
\end{eqnarray}
Similar argument to (\ref{MDP V 06*}), we have
\begin{eqnarray*}
& &\sum_{i=1}^{M}(\lambda_{i}^2-\lambda_{i})(G(u^{0}(t)+\sqrt{\e}\lambda(\e)X^{M,h^{\e}}(t),t),e_i)
(G(u^{0}(t)+\sqrt{\e}\lambda(\e)X^{M,h^{\e}}(t),t),e_i)'\nonumber\\
&\leq&
C\|u^{0}(t)+\sqrt{\e}\lambda(\e)X^{M,h^{\e}}(t)\|_{\mathbb{V}}^2.
\end{eqnarray*}
Thus
\begin{eqnarray}\label{MDP * 05}
J_{5}(T)
&\leq&
C_pE\Big[\int_0^{T\wedge\tau_{J}}\|X^{M,h^{\e}}(t)\|_{\ast}^{p-2}\|u^{0}(t)+\sqrt{\e}\lambda(\e)X^{M,h^{\e}}(t)\|_{\mathbb{V}}^2dt\Big]\nonumber\\
&\leq&
C_p\int_0^{T}E\Big[\sup_{s\in[0,t\wedge\tau_{J}]}\|X^{M,h^{\e}}(s)\|_{\ast}^{p}\Big]dt\nonumber\\
& &\quad+C_{p,T}E\Big[\sup_{t\in[0,T]}\big(\|u^{0}(t)\|_{\mathbb{V}}^{p}+
\varepsilon^{\frac{p}{2}}\lambda^p(\varepsilon)\|X^{M,h^{\e}}(t)\|_{\mathbb{V}}^{p}\big)\Big].
\end{eqnarray}
Using the B-D-G, Cauchy-Schwardz and Young inequalities and (\ref{G-01}), (\ref{G-02}), (\ref{curl}),
\begin{eqnarray}\label{MDP * 06}
J_{6}(T)
&\leq&
C_pE\Big[\int_0^{T\wedge\tau_{J}}\|X^{M,h^{\e}}(t)\|_{\ast}^{2p-2}\|u^{0}(t)+\sqrt{\e}\lambda(\e)X^{M,h^{\e}}(t)\|_{\mathbb{V}}^2dt\Big]^{\frac{1}{2}}\nonumber\\
&\leq&
C_pE\Big[\sup_{t\in[0,T\wedge\tau_{J}]}\|X^{M,h^{\e}}(t)\|_{\ast}^{p-1}
\big(\int_0^{T\wedge\tau_{J}}\|u^{0}(t)+\sqrt{\e}\lambda(\e)X^{M,h^{\e}}(t)\|_{\mathbb{V}}^2dt\big)^{\frac{1}{2}}\Big]\nonumber\\
&\leq&
\kappa E\Big[\sup_{t\in[0,T\wedge\tau_{J}]}\|X^{M,h^{\e}}(t)\|_{\ast}^p\Big]\nonumber\\
& &\quad+C_{p,\kappa,T}E\Big[\sup_{t\in[0,T]}\big(\|u^{0}(t)\|_{\mathbb{V}}^{p}+
\varepsilon^{\frac{p}{2}}\lambda^p(\varepsilon)\|X^{M,h^{\e}}(t)\|_{\mathbb{V}}^{p}\big)\Big],
\end{eqnarray}
and
\begin{eqnarray}\label{MDP * 07}
J_{7}(T)
&\leq&
\frac{C_p}{\lambda^2(\e)}E\Big[\int_0^{T\wedge\tau_{J}}\|X^{M,h^{\e}}(t)\|_{\ast}^{p-2}
\|u^{0}(t)+\sqrt{\e}\lambda(\e)X^{M,h^{\e}}(t)\|_{\mathbb{V}}^2dt\Big]\nonumber\\
&\leq&
C_p\int_0^{T}E\Big[\sup_{s\in[0,t\wedge\tau_{J}]}\|X^{M,h^{\e}}(s)\|_{\ast}^{p}\Big]dt\nonumber\\
& &\quad+C_{p,T}E\Big[\sup_{t\in[0,T]}\big(\|u^{0}(t)\|_{\mathbb{V}}^{p}+
\varepsilon^{\frac{p}{2}}\lambda^p(\varepsilon)\|X^{M,h^{\e}}(t)\|_{\mathbb{V}}^{p}\big)\Big].
\end{eqnarray}
Combining with (\ref{MDP * 00}) and (\ref{MDP * 01})--(\ref{MDP * 07}), we have
\begin{eqnarray}\label{MDP * 08}
& &(1-\delta-\eta-\kappa)E\Big[\sup_{t\in[0,T\wedge\tau_{J}]}\|X^{M,h^{\e}}(t)\|_{\ast}^p\Big]
+\frac{\nu p}{\alpha}E\Big[\int_0^{T\wedge\tau_{J}}\|X^{M,h^{\e}}(t)\|_{\ast}^{p}dt\Big]\nonumber\\
&\leq&
C_p\int_0^{T}E\Big[\sup_{s\in[0,t\wedge\tau_{J}]}\|X^{M,h^{\e}}(s)\|_{\ast}^{p}\Big]dt\nonumber\\
& &\quad+C_{p,\delta,\eta,\kappa,N,T}\Big\{\sup_{t\in[0,T]}\|u^{0}(t)\|_{\mathbb{V}}^{p}+E\Big[\sup_{t\in[0,T]}\|X^{M,h^{\e}}(t)\|_{\mathbb{V}}^{p}\big)\Big]\Big\}.
\end{eqnarray}
Let $\delta=\eta=\kappa=\frac{1}{4}$. The Gronwall lemma and (\ref{MDP * 08}) imply that
\begin{eqnarray}\label{MDP *}
E\Big[\sup_{t\in[0,T\wedge\tau_{J}]}\|X^{M,h^{\e}}(t)\|_{\ast}^p\Big]\leq C_{p,N}.
\end{eqnarray}
Letting $J\rightarrow\infty$ , we obtain (\ref{MDP Estimation 03}).
\end{proof}

\vskip 0.2cm

Let $\mathbb{K}$ be a separable Hilbert space. Given $p>1$, $\beta\in(0,1)$, let $W^{\beta,p}([0,T];\mathbb{K})$ be the
 space of all $u\in L^p([0,T];\mathbb{K})$ such that
$$
\int_0^T\int_0^T\frac{\|u(t)-u(s)\|^p_\mathbb{K}}{|t-s|^{1+\beta p}}dtds<\infty,
$$
endowed with the norm
$$
\|u\|^p_{W^{\beta,p}([0,T];\mathbb{K})}:=\int_0^T\|u(t)\|^p_{\mathbb{K}}dt+\int_0^T\int_0^T\frac{\|u(t)-u(s)\|^p_\mathbb{K}}{|t-s|^{1+\beta p}}dtds.
$$

The following result is a variant of the criteria for compactness proved in \cite{Lions} (Sect. 5, Ch. I)
 and \cite{Temam 1983} (Sect. 13.3).
\begin{lem}\label{Compact}{\rm
Let $\mathbb{K}_0\subset \mathbb{K}\subset \mathbb{K}_1$ be Banach spaces, $\mathbb{K}_0$ and $\mathbb{K}_1$ reflexive, with compact embedding of $\mathbb{K}_0$ into $\mathbb{K}$.
For $p\in(1,\infty)$ and $\beta\in(0,1)$, let $\Lambda$ be the space
$$
\Lambda=L^p([0,T];\mathbb{K}_0)\cap W^{\beta,p}([0,T];\mathbb{K}_1)
$$
endowed with the natural norm. Then the embedding of $\Lambda$ into $L^p([0,T];\mathbb{K})$ is compact.
}\end{lem}

\vskip 0.2cm

We will apply the above criteria to prove the following result.
\begin{prp}\label{Prop Tight}
$\{X^{h^{\e}}\}$ is tight in $L^2([0,T];\mathbb{V})$.
\end{prp}
\begin{proof}
Note that
\begin{eqnarray*}
X^{h^{\varepsilon}}(t)
&=&-\int_0^t\nu\widehat{A}X^{h^{\varepsilon}}(s)ds-\int_0^t\widehat{B}(X^{h^{\varepsilon}}(s),u^{0}(s)
+\sqrt{\varepsilon}\lambda(\varepsilon)X^{h^{\varepsilon}}(s))ds-\int_0^t\widehat{B}(u^{0}(s),X^{h^{\varepsilon}}(s))ds\nonumber\\
& &+\int_0^t\frac{1}{\sqrt{\e}\lambda(\varepsilon)}
(\widehat{F}(u^0(s)+\sqrt{\varepsilon}\lambda(\varepsilon)X^{h^{\varepsilon}}(s),s)-\widehat{F}(u^0(s),s))ds\nonumber\\
& &+\int_0^t\frac{1}{\lambda(\e)}\widehat{G}(u^{0}(s)+\sqrt{\varepsilon}\lambda(\varepsilon)X^{h^{\varepsilon}}(s),s)dW(s)
+\int_0^t\widehat{G}(u^{0}(s)+\sqrt{\varepsilon}\lambda(\varepsilon)X^{h^{\varepsilon}}(s),s)\dot{h}^{\e}(s)ds\nonumber\\
&:=&I_{1}(t)+I_{2}(t)+I_{3}(t)+I_{4}(t)+I_{5}(t)+I_{6}(t).
\end{eqnarray*}
By Lemma \ref{MDP lemma 01}, we have
\begin{equation}\label{prop 02}
E\Big[\int_0^T\|X^{h^{\e}}(t)\|^2_{\mathbb{W}}dt\Big]\leq C_{2,N}
\end{equation}
where $C_{2,N}$ is a constant independent of $\e$. We next prove
\begin{equation}\label{prop 03}
\sup_{\e\in(0,\e_0)}E\Big[\|X^{h^{\e}}\|^2_{W^{\beta,2}([0,T],\mathbb{W}^{\ast})}\Big]\leq C_\beta<\infty, \ \ \beta\in(0,1/2)
\end{equation}
here $\e_0$ is the constant stated in Lemma \ref{MDP lemma 01}.

\vskip 0.2cm

Noting that, for any $u\in\mathbb{W}$ and $v\in\mathbb{V}$,
\begin{eqnarray*}
(\widehat{A}u,v)_\mathbb{V}=((u,v)),
\end{eqnarray*}
and $\|v\|^2_\mathbb{V}\geq \alpha\|v\|^2$ (see (\ref{a})), we have
\begin{eqnarray}\label{Estation-A}
\|\widehat{A}u\|_\mathbb{V}
=
\sup_{\|v\|_\mathbb{V}\leq1}|(\widehat{A}u,v)_\mathbb{V}|
=
\sup_{\|v\|_\mathbb{V}\leq1}|((u,v))|
\leq
\|u\|\sup_{\|v\|_\mathbb{V}\leq1}\|v\|
\leq
\alpha^{-1/2}\|u\|.
\end{eqnarray}
Then
\begin{eqnarray}\label{tight *01}
\|I_1(t)-I_1(s)\|^2_\mathbb{V}
&=&
\|\int_s^t\nu \widehat{A}X^{h^\e}(l)dl\|^2_\mathbb{V}
\leq
\int_s^t\|\nu \widehat{A}X^{h^\e}(l)\|^2_\mathbb{V}dl(t-s)\nonumber\\
&\leq&
\nu^2\alpha^{-1}\int_s^t\|X^{h^\e}(l)\|^2dl(t-s)\nonumber\\
&\leq&
\nu^2\alpha^{-2}\sup_{l\in[0,T]}\|X^{h^\e}(l)\|^2_\mathbb{V}(t-s)^2.
\end{eqnarray}
By (\ref{tight *01}), we have
\begin{eqnarray}\label{tight 01}
E\Big[\|I_1\|^2_{\mathbb{W}^{\beta,2}([0,T];\mathbb{V})}\Big]
&=&
E\Big[\int_0^T\|I_1(s)\|^2_\mathbb{V}ds
+
\int_0^T\int_0^T\frac{\|I_1(t)-I_1(s)\|^2_\mathbb{V}}{|t-s|^{1+2\beta}}dsdt\Big]\nonumber\\
&\leq&
TE\Big[\sup_{s\in[0,T]}\|I_1(s)\|^2_\mathbb{V}\Big]
+
E\int_0^T\int_0^T\nu^2\alpha^{-2}\sup_{l\in[0,T]}\|X^{h^\e}(l)\|^2_\mathbb{V}(t-s)^{1-2\beta}dsdt\nonumber\\
&\leq&
C_{\beta,T}E\Big[\sup_{l\in[0,T]}\|X^{h^\e}(l)\|^2_\mathbb{V}\Big],\ \ \ \beta\in(0,1).
\end{eqnarray}
By Lemma \ref{Lem-B-01},

\begin{eqnarray*}
\|I_2(t)-I_2(s)\|^2_{{\mathbb{W}}^{\ast}}
&=&
\|\int_s^t\widehat{B}(X^{h^{\varepsilon}}(l),u^{0}(l)+\sqrt{\varepsilon}\lambda(\varepsilon)X^{h^{\varepsilon}}(l))dl\|^2_{{\mathbb{W}}^{\ast}}\nonumber\\
&\leq&
\int_s^t\|\widehat{B}(X^{h^{\varepsilon}}(l),u^{0}(l)+\sqrt{\varepsilon}\lambda(\varepsilon)X^{h^{\varepsilon}}(l))\|^2_{{\mathbb{W}}^{\ast}}dl\cdot(t-s)\nonumber\\
&\leq&
\int_s^t C\|X^{h^{\varepsilon}}(l)\|^2_{\mathbb{W}}\|u^{0}(l)+\sqrt{\varepsilon}\lambda(\varepsilon)X^{h^{\varepsilon}}(l))\|^2_{\mathbb{V}}dl\cdot(t-s)\nonumber\\
&\leq&
C\sup_{l\in[0,T]}\|X^{h^{\e}}(l)\|^2_{\mathbb{W}}\sup_{l\in[0,T]}
\big(\|u^{0}(l)\|^2_{\mathbb{V}}+\varepsilon\lambda^2(\varepsilon)\|X^{h^{\e}}(l)\|^2_{\mathbb{V}}\big)(t-s)^2\nonumber\\
&\leq&
C\sup_{l\in[0,T]}\|u^{0}(l)\|^2_{\mathbb{V}}\sup_{l\in[0,T]}\|X^{h^{\varepsilon}}(l))\|^2_{\mathbb{W}}(t-s)^2\nonumber\\
& &\quad+C\sup_{l\in[0,T]}\|X^{h^{\varepsilon}}(l))\|^4_{\mathbb{W}}(t-s)^2
+C\varepsilon^2\lambda^4(\varepsilon)\|X^{h^{\varepsilon}}(l))\|^4_{\mathbb{V}}(t-s)^2,
\end{eqnarray*}
and
\begin{eqnarray*}
\|I_3(t)-I_3(s)\|^2_{{\mathbb{W}}^{\ast}}
&=&
\|\int_s^t\widehat{B}(u^{0}(l),X^{h^{\varepsilon}}(l))dl\|^2_{{\mathbb{W}}^{\ast}}\nonumber\\
&\leq&
\int_s^t\|\widehat{B}(u^{0}(l),X^{h^{\varepsilon}}(l))\|^2_{{\mathbb{W}}^{\ast}}dl\cdot(t-s)\nonumber\\
&\leq&
\int_s^t C\|u^{0}(l)\|^2_{\mathbb{W}}\|X^{h^{\varepsilon}}(l))\|^2_{\mathbb{V}}dl\cdot(t-s)\nonumber\\
&\leq&
C\sup_{l\in[0,T]}\|u^{0}(l)\|^2_{\mathbb{W}}\sup_{l\in[0,T]}\|X^{h^{\varepsilon}}(l))\|^2_{\mathbb{V}}(t-s)^2,
\end{eqnarray*}
which yield, for $\beta\in(0,1)$
\begin{eqnarray}\label{tight 02}
& &E\Big[\|I_2\|^2_{\mathbb{W}^{\beta,2}([0,T];\mathbb{W}^*)}\Big]\nonumber\\
&\leq&
C_{\beta,T}E\Big[\sup_{l\in[0,T]}\|X^{h^\e}(l)\|^4_\mathbb{W}\Big]
+C_{\beta,T}\sup_{l\in[0,T]}\|u^{0}(l)\|^2_{\mathbb{V}}E\Big[\sup_{l\in[0,T]}\|X^{h^\e}(l)\|^2_\mathbb{W}\Big],
\end{eqnarray}
and
\begin{eqnarray}\label{tight 03}
E\Big[\|I_3\|^2_{\mathbb{W}^{\beta,2}([0,T];\mathbb{W}^*)}\Big]
\leq
C_{\beta,T}\sup_{l\in[0,T]}\|u^{0}(l)\|^2_{\mathbb{W}}E\Big[\sup_{l\in[0,T]}\|X^{h^{\varepsilon}}(l))\|^2_{\mathbb{V}}\Big].
\end{eqnarray}
By (\ref{F-01})--(\ref{F-04}),
\begin{eqnarray*}
\|I_4(t)-I_4(s)\|^2_\mathbb{V}
&=&
\|\frac{1}{\sqrt{\e}\lambda(\e)}\int_s^t
    (\widehat{F}(u^0(l)+\sqrt{\e}\lambda(\varepsilon)X^{h^{\e}}(l),l)-\widehat{F}(u^0(l),l))dl\|^2_{\mathbb{V}}\nonumber\\
&\leq&
\frac{1}{\sqrt{\e}\lambda(\e)}\int_s^t\|
    (\widehat{F}(u^0(l)+\sqrt{\e}\lambda(\e)X^{h^{\e}}(l),l)-\widehat{F}(u^0(l),l))\|^2_{\mathbb{V}}dl\cdot(t-s)\nonumber\\
&\leq&
C\int_s^t\|X^{h^{\varepsilon}}(l)\|^2_{\mathbb{V}}dl\cdot(t-s)\nonumber\\
&\leq&
C\sup_{l\in[0,T]}\|X^{h^{\varepsilon}}(l)\|^2_{\mathbb{V}}(t-s)^2,
\end{eqnarray*}
which implies
\begin{eqnarray}\label{tight 04}
E\Big[\|I_4\|^2_{\mathbb{W}^{\beta,2}([0,T];\mathbb{W}^*)}\Big]
\leq
C_{\beta,T}E\Big[\sup_{l\in[0,T]}\|X^{h^{\varepsilon}}(l)\|^2_{\mathbb{V}}\Big],\ \ \beta\in(0,1).
\end{eqnarray}
By (\ref{G-01}), (\ref{Jian G}) and using B-D-G, H$\ddot{o}$lder's inequalities,
\begin{eqnarray*}
E\Big[\|I_5(t)-I_5(s)\|^{2p}_\mathbb{V}\Big]
&=&
E\Big[\frac{1}{\lambda^2(\e)}\int_s^t\|\widehat{G}(u^{0}(l)+\sqrt{\varepsilon}\lambda(\varepsilon)X^{h^{\varepsilon}}(l),l)\|^2_{{\mathbb{V}}^{\otimes m}}dl\Big]^p\nonumber\\
&\leq&
\frac{C_p}{\lambda^{2p}(\e)}E\Big[\int_s^t\|u^{0}(l)+\sqrt{\varepsilon}\lambda(\varepsilon)X^{h^{\varepsilon}}(l)\|^2_{\mathbb{V}}dl\Big]^p\nonumber\\
&\leq&
C_p\sup_{l\in[0,T]}\|u^{0}(l)\|^{2p}_{\mathbb{V}}(t-s)^p+C_p\varepsilon E\Big[\sup_{l\in[0,T]}\|X^{h^{\varepsilon}}(l)\|^{2p}_{\mathbb{V}}\Big](t-s)^p.
\end{eqnarray*}
Hence
\begin{eqnarray}\label{tight 05}
&&E\Big[\|I_5\|^{2p}_{\mathbb{W}^{\beta,2p}([0,T];\mathbb{W}^*)}\Big]\\
&\leq&
C_{\beta,T,p}\Big(\sup_{l\in[0,T]}\|u^{0}(l)\|^{2p}_{\mathbb{V}}+E\Big[\sup_{l\in[0,T]}\|X^{h^{\varepsilon}}(l))\|^{2p}_{\mathbb{V}}\Big]\Big)\nonumber\\
&&\cdot\Big(1+\int_0^T\int_0^T|t-s|^{(1-2\beta)p-1}dsdt\Big)
\leq C_{\beta,T,p}<\infty,\ \ \beta\in(0,1/2),\ (1-2\beta)p>0.\nonumber
\end{eqnarray}

By (\ref{G-01}), (\ref{Jian G}) and H$\ddot{o}$lder's inequality again,
\begin{eqnarray*}
\|I_6(t)-I_6(s)\|^2_\mathbb{V}
&=&
\|\int_s^t\widehat{G}(u^{0}(l)+\sqrt{\varepsilon}\lambda(\varepsilon)X^{h^{\varepsilon}}(l),l)\dot{h}^{\e}(l)dl\|^2_{\mathbb{V}}\nonumber\\
&\leq&
\int_s^t\|\widehat{G}(u^{0}(l)+\sqrt{\varepsilon}\lambda(\varepsilon)X^{h^{\varepsilon}}(l),l)\|^2_{{\mathbb{V}}^{\otimes m}}dl\int_s^t\|\dot{h}^{\e}(l)\|^2_{{\mathbb{R}}^m}dl\nonumber\\
&\leq&
C\int_s^t\|u^{0}(l)+\sqrt{\varepsilon}\lambda(\varepsilon)X^{h^{\varepsilon}}(l)\|^2_{\mathbb{V}}dl\int_s^t\|\dot{h}^{\e}(l)\|^2_{{\mathbb{R}}^m}dl\nonumber\\
&\leq&
C_{T}\Big\{\sup_{l\in[0,T]}\|u^{0}(l)\|^2_{\mathbb{V}}+\|X^{h^{\varepsilon}}(l)\|^2_{\mathbb{V}}\Big\}\int_s^t\|\dot{h}^{\e}(l)\|^2_{{\mathbb{R}}^m}dl(t-s),
\end{eqnarray*}
which implies, by Fubini Theorem,
\begin{eqnarray}\label{tight 06}
E\Big[\|I_6\|^2_{\mathbb{W}^{\beta,2}([0,T];\mathbb{W}^*)}\Big]
\leq
C_{\beta,N,T}\Big\{\sup_{l\in[0,T]}\|u^{0}(l)\|^2_{\mathbb{V}}
+E\Big[\sup_{l\in[0,T]}\|X^{h^{\varepsilon}}(l)\|^2_{\mathbb{V}}\Big]\Big\},\ \ \beta\in(0,1).
\end{eqnarray}
Combining (\ref{tight 01})--(\ref{tight 06}) and (\ref{MDP Estimation 01}), we obtain (\ref{prop 03}).\\
Since the imbedding $\mathbb{W}\subset\mathbb{V}$ is compact, by Lemma \ref{Compact},
$$\Lambda=L^2([0,T],\mathbb{W})\cap\mathbb{W}^{\beta,2}([0,T],\mathbb{W}^*),\ \ \beta\in(0,1/2)$$
is compactly imbedded in $L^2([0,T],\mathbb{V})$. Denote $\|\cdot\|_\Lambda:=\|\cdot\|_{L^2([0,T],\mathbb{W})}+\|\cdot\|_{\mathbb{W}^{\beta,2}([0,T],\mathbb{W}^*)}$. Thus for any $L>0$,
$$K_L=\{u\in L^2([0,T],\mathbb{V}),\ \|u\|_\Lambda\leq L\}$$
is relatively compact in $L^2([0,T],\mathbb{V})$.
We have $$P(X^{h^\e}\not\in K_L)\leq P(\|X^{h^\e}\|_\Lambda\geq L)\leq \frac{1}{L}E(\|X^{h^\e}\|_\Lambda)\leq \frac{C}{L}.$$
Freely choosing the constant $L$, we see that $\{X^{h^\e},\ \e>0\}$ is tight in $L^2([0,T],\mathbb{V})$.
\end{proof}
\vskip 0.6cm

Using similar arguments as in the proof of Theorem 2.2 in \cite{FG95}, we see that the imbedding $C([0,T],\mathbb{V})\cap\Big(\mathbb{W}^{\beta_1,p_1}([0,T],\mathbb{W}^*)+\cdots +\mathbb{W}^{\beta_m,p_m}([0,T],\mathbb{W}^*)\Big)\subset C([0,T];\mathbb{W}^*)$ is compact if $\beta_i p_i>1$, $i=1, 2, \cdots, m$. The following result is a
consequence of (\ref{tight 01}) (\ref{tight 02}) (\ref{tight 03}) (\ref{tight 04}) (\ref{tight 05}) and (\ref{tight 06}).
\begin{prp}\label{Prop tight W}
$\{X^{h^\e}\}$ is tight in $C([0,T];\mathbb{W}^*)$.
\end{prp}

\vskip 0.2cm
We are ready now to verify the condition (a) in Theorem \ref{thm BD}.
\vskip 0.2cm

\begin{thm}\label{Thm condition 02}
For every fixed $N\in\mathbb{N}$, let $h^\e,\ h\in\mathcal{A}_N$ be such that $h^\e$ converges in distribution to $h$ as $\e\rightarrow0$.
Then
$$
\Gamma^\e\left(W(\cdot)+{\lambda(\e)}\int_0^{\cdot}\dot h^\e(s)ds\right)\text{ converges in distribution to }\Gamma^0(\int_0^{\cdot}\dot h(s)ds)
$$
in $C([0,T];\mathbb{V})$ as $\e\rightarrow0$.
\end{thm}

\begin{proof}
Note that $X^{h^\e}=\Gamma^\e\left(W(\cdot)+{\lambda(\e)}\int_0^{\cdot}\dot h^\e(s)ds\right)$. By Proposition \ref{Prop Tight} and Proposition \ref{Prop tight W}, we know that
$\{X^{h^\e}\}$ is tight in $L^2([0,T],\mathbb{V})\cap C([0,T],\mathbb{W}^*)$.

\vskip 0.2cm
Let $(X,\ h,\ W)$ be any limit point of the tight family $\{(X^{h^\e},\ h^\e,\ W),\ \varepsilon\in(0,\varepsilon_0)\}$ with respect to the convergence in law.
We must show that $X$ has the same law as $\Gamma^0(\int_0^{\cdot}\dot h(s)ds)$,
and that actually $X^{h^\e}\Longrightarrow X$ in the smaller space $C([0,T];\mathbb{V})$.
\vskip 0.2cm

Set
$$
\Pi=\Big(L^2([0,T],\ \mathbb{V})\cap C([0,T],\ \mathbb{W}^*),\ S_N,\ C([0,T],\ \mathbb{R}^m)\Big).
$$
By Skorokhod representation theorem, there exist a stochastic basis
$(\Omega^1,\mathcal{F}^1,\{\mathcal{F}_t^1\}_{t\in[0,T]},\mathbb{P}^1)$ and, on this basis,
$\Pi$-valued random variables $(\widetilde{X}^\varepsilon,\ \widetilde{h}^\e,\ \widetilde{W}^\varepsilon)$, $(\widetilde{X},\ \widetilde{h},\ \widetilde{W})$
such that $(\widetilde{X}^\varepsilon,\ \widetilde{h}^\e,\ \widetilde{W}^\varepsilon)$ (respectively $(\widetilde{X},\ \widetilde{h},\ \widetilde{W})$) has the same law as $\{(X^{h^\e},\ h^\e,\ W),\ \varepsilon\in(0,\varepsilon_0)\}$ (respectively $(X,\ h,\ W)$), and
$(\widetilde{X}^\varepsilon,\ \widetilde{h}^\e,\ \widetilde{W}^\varepsilon)\rightarrow (\widetilde{X},\ \widetilde{h},\ \widetilde{W})$-$\mathbb{P}^1$ a.s. in $\Pi$.
\vskip 0.2cm

From the equation satisfied by $(X^{h^\e},\ h^\e,\ W)$, we see that $(\widetilde{X}^\varepsilon,\ \widetilde{h}^\e,\ \widetilde{W}^\varepsilon)$ satisfies the following integral equation in $\mathbb{W}^*$
\begin{eqnarray}\label{Eq X01}
\widetilde{X}^\varepsilon(t)
&=&-\nu\int_0^t\widehat{A}\widetilde{X}^\varepsilon(s)ds-\int_0^t\widehat{B}(\widetilde{X}^\varepsilon(s),u^{0}(s)
+\sqrt{\varepsilon}\lambda(\varepsilon)\widetilde{X}^\varepsilon(s))ds-\int_0^t\widehat{B}(u^{0}(s),\widetilde{X}^\varepsilon(s))ds\nonumber\\
& &+\int_0^t\frac{1}{\sqrt{\e}\lambda(\varepsilon)}
(\widehat{F}(u^0(s)+\sqrt{\varepsilon}\lambda(\varepsilon)\widetilde{X}^\varepsilon(s),s)-\widehat{F}(u^0(s),s))ds\nonumber\\
& &+\int_0^t\frac{1}{\lambda(\e)}\widehat{G}(u^{0}(s)+\sqrt{\varepsilon}\lambda(\varepsilon)\widetilde{X}^\varepsilon(s),s)d\widetilde{W}^\varepsilon(s)
+\int_0^t\widehat{G}(u^{0}(s)+\sqrt{\varepsilon}\lambda(\varepsilon)\widetilde{X}^\varepsilon(s),s)\dot{\widetilde{h}^\e}(s)ds,\nonumber\\
\end{eqnarray}
and moreover,(see (\ref{MDP Estimation 01}))
\begin{eqnarray}\label{Estation-X01}
\sup_{\varepsilon\in(0,\varepsilon_0)}E^1\Big[\sup_{s\in[0,T]}\|\widetilde{X}^\varepsilon(s)\|^p_\mathbb{W}\Big]\leq C_{p,N},\ \text{for any }2\leq p<\infty,
\end{eqnarray}
where $E^1$ stands for the expectation under the probability measure $\mathbb{P}^1$.
Owing to (\ref{Estation-X01}) and the fact that $\widetilde{X}^\e\rightarrow \widetilde{X}$ $\mathbb{P}^1$-a.s. in $L^2([0,T],\ \mathbb{V})\cap C([0,T],\ \mathbb{W}^*)$, we can assert that\\
there exists a sub-sequence $\widetilde{X}^{\e_k}$
 such that, as $\e_k\rightarrow 0$
\begin{itemize}
\item[ ]
 (a) $\widetilde{X}^{\e_k}\rightarrow \widetilde{X}$ weakly-* in $L^p(\Omega^1,\mathcal{F}^1,\mathbb{P}^1,\ L^\infty([0,T],\mathbb{W}))$ for any $2\leq p<\infty$,
\end{itemize}
Moreover,
\begin{eqnarray}\label{eq 05}
E^1\Big[\sup_{s\in[0,T]}\|\widetilde{X}(s)\|^p_\mathbb{W}\Big]\leq C_{p,N},\ \text{for any }2\leq p<\infty
\end{eqnarray}

\begin{itemize}
\item[ ] (b) $\widetilde{X}^{\e_k}\rightarrow \widetilde{X}$ weakly in $L^p(\Omega^1,\mathcal{F}^1,\mathbb{P}^1,\ L^q([0,T],\mathbb{V}))$ for any $2\leq p,\ q<\infty$.
\end{itemize}
Thanks to (\ref{Estation-X01}),(\ref{eq 05}), and the fact that $\widetilde{X}^\e\rightarrow \widetilde{X}$ $\mathbb{P}^1$-a.s. in $L^2([0,T],\ \mathbb{V})\cap C([0,T],\ \mathbb{W}^*)$, we have, as $\e_k\rightarrow 0$
\begin{itemize}
\item[ ](c) $\widetilde{X}^{\e_k}-\widetilde{X}\rightarrow 0$ in $L^2(\Omega^1,\mathcal{F}^1,\mathbb{P}^1,\ L^2([0,T],\mathbb{V}))$,


\item[ ](d) $\widehat{A}\widetilde{X}^{\e_k}\rightarrow \widehat{A}\widetilde{X}$ weakly in $L^2(\Omega^1,\mathcal{F}^1,\mathbb{P}^1,\ L^2([0,T],\mathbb{V}))$,


 \item[ ](e) $\frac{1}{\lambda(\e)}\int_0^\cdot\widehat{G}(u^{0}(s)+\sqrt{\varepsilon}\lambda(\varepsilon)\widetilde{X}^\epsilon(s),s)d\widetilde{W}^\varepsilon(s)
     \rightarrow 0$  in $L^2(\Omega^1,\mathcal{F}^1,\mathbb{P}^1,\ C([0,T],\mathbb{V}))$,


\item[ ] (f) $\frac{1}{\sqrt{\e}\lambda(\varepsilon)}
     (\widehat{F}(u^0(s)+\sqrt{\varepsilon}\lambda(\varepsilon)\widetilde{X}^\epsilon(s),s)-\widehat{F}(u^0(s),s))\rightarrow \widehat{F}'(u^0(s),s))\widetilde{X}(s)$ \\
      in $L^2(\Omega^1,\mathcal{F}^1,\mathbb{P}^1,\ L^2([0,T],\mathbb{V}))$,


\item[ ] (g) $\widehat{B}(\widetilde{X}^{\e_k}(s),u^{0}(s)
     +\sqrt{\varepsilon}\lambda(\varepsilon)\widetilde{X}^\varepsilon(s))\rightarrow \widehat{B}(\widetilde{X}(s),u^{0}(s))$ weakly\\
      in $L^2(\Omega^1,\mathcal{F}^1,\mathbb{P}^1,\ L^2([0,T],\mathbb{W}^*))$,


\item[ ] (h) $\widehat{B}(u^{0}(s),\widetilde{X}^{\e_k}(s))\rightarrow \widehat{B}(u^{0}(s),\widetilde{X}(s))$ in $L^2(\Omega^1,\mathcal{F}^1,\mathbb{P}^1,\ L^2([0,T],\mathbb{W}^*))$,


\item[ ] (i) $\widehat{G}(u^{0}(s)+\sqrt{\varepsilon}\lambda(\varepsilon)\widetilde{X}^\epsilon(s),s)\dot{\widetilde{h}^\e}(s)\rightarrow \widehat{G}(u^{0}(s),s)\dot{\widetilde{h}}(s)$ weakly \\
    in $L^2(\Omega^1,\mathcal{F}^1,\mathbb{P}^1,\ L^2([0,T],\mathbb{V}))$.

\end{itemize}
Letting $\e_k\rightarrow0$ in (\ref{Eq X01}) and using (a)--(i), it is easy to see that $\widetilde{X}$ is the
unique solution of the following equation
\begin{eqnarray}\label{Eq X02}
\widetilde{X}(t)&=&-\nu\int_0^t\widehat{A}\widetilde{X} (s)ds-\int_0^t\widehat{B}(\widetilde{X} (s),u^0 (s))ds-\int_0^t\widehat{B}(u^0 (s),\widetilde{X} (s))ds\\
& &+
\int_0^t\widehat{F}'(u^0 (s),s)\widetilde{X} (s)ds
+
\int_0^t\widehat{G}(u^0 (s),s)\dot{\widetilde{h}} (s) ds.\nonumber
\end{eqnarray}
This implies
$$
X\overset{law}{=}\widetilde{X}=\Gamma^0(\int_0^\cdot \dot{\tilde{h}}(s)ds)\overset{law}{=}\Gamma^0(\int_0^\cdot \dot{{h}}(s)ds).
$$

Next, we will prove the following stronger statement:
\begin{eqnarray}\label{Eq F01}
\lim_{\varepsilon\rightarrow0}\sup_{t\in[0,T]}\|\widetilde{X}(t)-\widetilde{X}^\e(t)\|_\mathbb{V}=0,\ \ \ \text{in probability}.
\end{eqnarray}
Because $\widetilde{X}^\e{=}X^{h^\e}$ in law, (\ref{Eq F01}) implies that $X^{h^\e}\Rightarrow X$ in the space $C([0,T],\mathbb{V})$.

\vskip 0.3cm
Let $v^\e(t)=\widetilde{X}^\e(t)-\widetilde{X}(t)$. Using ${\rm It\hat{o}}$'s formula, we have
\begin{eqnarray}\label{V 00}
& &\|v^\e(t)\|^2_{\mathbb{V}}+2\nu\int_0^t\|v^\e(s)\|^2ds\nonumber\\
&=&
   -2\int_0^t\langle\widehat{B}(\widetilde{X}^\e(s),u^0(s)+\sqrt{\e}\lambda(\e)\widetilde{X}^\e(s))-\widehat{B}(\widetilde{X}(s),u^0(s)),
         v^\e(s)\rangle ds\nonumber\\
& &-2\int_0^t\langle\widehat{B}(u^0(s),\widetilde{X}^\e(s))-\widehat{B}(u^0(s),\widetilde{X}(s)),v^\e(s)\rangle ds\nonumber\\
& &+2\int_0^t\big(\frac{1}{\sqrt{\e}\lambda(\e)}(\widehat{F}(u^0(s)+\sqrt{\varepsilon}\lambda(\varepsilon)\widetilde{X}^\epsilon(s),s)-\widehat{F}(u^0(s),s))
         -\widehat{F}'(u^0(s),s)\cdot\widetilde{X}(s),v^\e(s)\big)_{\mathbb{V}}ds\nonumber\\
& &+2\int_0^t\frac{1}{\lambda^2(\e)}\|\widehat{G}\big(u^0(s)+\sqrt{\e}\lambda(\e)\widetilde{X}^\e(s)\big)\|_{{\mathbb{V}}^\otimes m}^2ds\nonumber\\
& &+2\int_0^t\frac{1}{\lambda(\e)}\big(\widehat{G}(u^0(s)+\sqrt{\e}\lambda(\e)\widetilde{X}^\e(s),s),v^\e(s)\big)_\mathbb{V}d\widetilde{W}^\e(s)\nonumber\\
& &+2\int_0^t\big(\widehat{G}(u^0(s)+\sqrt{\e}\lambda(\e)\widetilde{X}^\e(s),s)\dot{\widetilde{h}^\e}(s)
-\widehat{G}(u^0(s),s)\dot{\widetilde{h}}(s),v^\e(s)\big)_\mathbb{V}ds.
\end{eqnarray}
By Lemma \ref{Lem-B-01}, we have
\begin{eqnarray}\label{V 01}
& &\big|\big\langle\widehat{B}(\widetilde{X}^\e(s),u^0(s)+\sqrt{\e}\lambda(\e)\widetilde{X}^\e(s))-\widehat{B}(\widetilde{X}(s),u^0(s)),
v^\e(s)\big\rangle\big|\nonumber\\
&\leq&
\big|\big\langle\widehat{B}(\widetilde{X}^\e(s),\sqrt{\e}\lambda(\e)\widetilde{X}^\e(s)),v^\e(s)\big\rangle\big|
+|\langle\widehat{B}(v^\e(s),u^0(s)),v^\e(s)\rangle|\nonumber\\
&\leq&
C_B\sqrt{\e}\lambda(\e)\|\widetilde{X}^\e(s)\|_{\mathbb{W}}^2\|v^\e(s)\|_{\mathbb{V}}+C_B\|v^\e(s)\|_{\mathbb{V}}^2\|u^0(s)\|_{\mathbb{W}},
\end{eqnarray}
and
\begin{eqnarray}\label{V 02}
\big\langle\widehat{B}(u^0(s),\widetilde{X}^\e(s))-\widehat{B}(u^0(s),\widetilde{X}(s)),v^\e(s)\big\rangle
=
\big\langle\widehat{B}(u^0(s),v^\e(s)),v^\e(s)\big\rangle=0.
\end{eqnarray}
By (\ref{F-01})--(\ref{F-04}), (\ref{Jian F1}), (\ref{Jian G}) and the mean value theorem,
\begin{eqnarray}\label{V 03}
& &\Big|\Big(\frac{1}{\sqrt{\e}\lambda(\e)}\big(\widehat{F}(u^0(s)+\sqrt{\varepsilon}\lambda(\varepsilon)\widetilde{X}^\epsilon(s),s)-\widehat{F}(u^0(s),s)\big)
         -\widehat{F}'(u^0(s),s)\widetilde{X}(s),v^\e(s)\Big)_{\mathbb{V}}\Big|\nonumber\\
&\leq&
C\sqrt{\varepsilon}\lambda(\varepsilon)\|\widetilde{X}^{\e}(s)\|_{\mathbb{V}}^2\|v^\e(s)\|_{\mathbb{V}}+C\|v^\e(s)\|_{\mathbb{V}}^2.
\end{eqnarray}
(\ref{G-02}) implies
\begin{eqnarray}\label{V 04}
\frac{1}{\lambda^2(\e)}\|\widehat{G}\big(u^0(s)+\sqrt{\e}\lambda(\e)\widetilde{X}^\e(s)\big)\|_{{\mathbb{V}}^\otimes m}^2
\leq
\frac{C}{\lambda^2(\e)}\|u^0(s)\|_{\mathbb{V}}^2+C_G\e\|\widetilde{X}^{\e}(s)\|_{\mathbb{V}}^2,
\end{eqnarray}
and
\begin{eqnarray}\label{V 05}
& &2\int_0^t\big|\big(\widehat{G}(u^0(s)+\sqrt{\e}\lambda(\e)\widetilde{X}^\e(s),s)\dot{\widetilde{h}^\e}(s)
-\widehat{G}(u^0(s),s)\dot{\widetilde{h}}(s),v^\e(s)\big)_\mathbb{V}\big|ds\nonumber\\
&\leq&
2\int_0^t\big|\big(\widehat{G}(u^0(s)+\sqrt{\e}\lambda(\e)\widetilde{X}^\e(s),s)\dot{\widetilde{h}^\e}(s)
-\widehat{G}(u^0(s),s)\dot{\widetilde{h}^\e}(s),v^\e(s)\big)_\mathbb{V}\big|ds\nonumber\\
& &+2\int_0^t\big|\big(\widehat{G}(u^0(s),s)\dot{\widetilde{h}^\e}(s)-\widehat{G}(u^0(s),s)\dot{\widetilde{h}}(s),v^\e(s)\big)_\mathbb{V}\big|ds\nonumber\\
&\leq&
C\sqrt{\e}\lambda(\e)\int_0^t\|\widetilde{X}^\e(s)\|_{\mathbb{V}}\|\dot{\widetilde{h}^\e}(s)\|_{\mathbb{R}^m}\|v^\e(s)\|_{\mathbb{V}}ds
+C\int_0^t\|u^0(s)\|_{\mathbb{V}}\|\dot{\widetilde{h}^\e}(s)-\dot{\widetilde{h}}(s)\|_{\mathbb{R}^m}\|v^\e(s)\|_{\mathbb{V}}ds\nonumber\\
&\leq&
\sqrt{\e}\lambda(\e)C_N\sup_{t\in[0,T]}\|\widetilde{X}^\e(s)\|_{\mathbb{V}}\big(\int_0^T\|v^\e(s)\|_{\mathbb{V}}^2ds\big)^\frac{1}{2}
+C_N\sup_{t\in[0,T]}\|u^0(s)\|_{\mathbb{V}}\big(\int_0^T\|v^\e(s)\|_{\mathbb{V}}^2ds\big)^\frac{1}{2}.\nonumber\\
\end{eqnarray}
Combining (\ref{V 00})--(\ref{V 05}), we get
\begin{eqnarray*}
& &\|v^\e(t)\|^2_{\mathbb{V}}+2\nu\int_0^t\|v^\e(s)\|^2ds\nonumber\\
&\leq&
\Big\{C\sqrt{\e}\lambda(\e)\big(\sup_{s\in[0,T]}\|\widetilde{X}^\e(s)\|_{\mathbb{W}}^2+\sup_{s\in[0,T]}\|\widetilde{X}^\e(s)\|_{\mathbb{V}}^2+1\big)
+C_N\sup_{t\in[0,T]}\|u^0(s)\|_{\mathbb{V}}\Big\}\Big(\int_0^T\|v^\e(s)\|_{\mathbb{V}}^2ds\Big)^\frac{1}{2}\nonumber\\
& &+C\int_0^t\|v^\e(s)\|_{\mathbb{V}}^2\big(\|u^0(s)\|_{\mathbb{W}}+1\big)ds\nonumber\\
& &+\frac{C}{\lambda^2(\e)}\sup_{s\in[0,T]}\|u^0(s)\|_{\mathbb{V}}^2+\e C\sup_{s\in[0,T]}\|\widetilde{X}^{\e}(s)\|_{\mathbb{V}}^2\nonumber\\
& &+\sup_{t\in[0,T]}\Big|\frac{2}{\lambda(\e)}\int_0^t\big(\widehat{G}(u^0(s)+\sqrt{\e}\lambda(\e)\widetilde{X}^\e(s),s),v^\e(s)\big)_\mathbb{V}
d\widetilde{W}^\e(s)\Big|.
\end{eqnarray*}

By Gronwall's inequality,
\begin{eqnarray}\label{Eq gronwall}
    \sup_{t\in[0,T]}\|v^\e(t)\|^2_{\mathbb{V}}
&\leq&
    \Theta(\e,T)
    \exp\Big(\int_0^T\varphi(s)ds\Big),
\end{eqnarray}

Here
\begin{eqnarray}\label{Eq varphi}
\varphi(s)=\|u^0(s)\|_\mathbb{W}+1,
\end{eqnarray}
and
\begin{eqnarray}
\Theta(\e,T)
&=&
\Big\{C\sqrt{\e}\lambda(\e)\big(\sup_{s\in[0,T]}\|\widetilde{X}^\e(s)\|_{\mathbb{W}}^2+\sup_{s\in[0,T]}\|\widetilde{X}^\e(s)\|_{\mathbb{V}}^2+1\big)\nonumber\\
& &\qquad\qquad+C_N\sup_{t\in[0,T]}\|u^0(s)\|_{\mathbb{V}}\Big\}\Big(\int_0^T\|v^\e(s)\|_{\mathbb{V}}^2ds\Big)^\frac{1}{2}\nonumber\\
& &+\frac{C}{\lambda^2(\e)}\sup_{s\in[0,T]}\|u^0(s)\|_{\mathbb{V}}^2+\e C\sup_{s\in[0,T]}\|\widetilde{X}^{\e}(s)\|_{\mathbb{V}}^2\nonumber\\
& &+\sup_{t\in[0,T]}\Big|\frac{2}{\lambda(\e)}\int_0^t\big(\widehat{G}(u^0(s)+\sqrt{\e}\lambda(\e)\widetilde{X}^\e(s),s),v^\e(s)\big)_\mathbb{V}
d\widetilde{W}^\e(s)\Big|.
\end{eqnarray}

Remembering that $\lim_{\e\rightarrow0}\widetilde{X}^\e=\widetilde{X}$ in $L^2([0,T],\mathbb{V})\ \ \mathbb{P}^1$-a.s.
and $\sup_{s\in[0,T]}\|u^0(s)\|_\mathbb{W}\leq C<\infty\ $,
we get
\begin{eqnarray}\label{Eq es 01}
 \exp\Big(\int_0^T\varphi(s)ds\Big)\leq C<\infty,
\end{eqnarray}
and
\begin{eqnarray}\label{Eq es 02}
 \lim_{\e\rightarrow0}\Theta(\e,T)
 =0,
      \ \ \ \text{in probability}.
\end{eqnarray}
Hence,
 \begin{eqnarray*}
  \sup_{t\in[0,T]}\|v^\e(t)\|^2_{\mathbb{V}}\rightarrow 0,\ \ \ \text{in probability}.
 \end{eqnarray*}
\end{proof}
Replacing $\frac{1}{\lambda(\e)}\int_0^t\widehat{G}(u^{h^\e}(s),s)dW(s)$ by 0 and replacing $h^\e$ by deterministic elements
in $\mathcal{H}_0$ in the proof of Proposition \ref{Prop Tight}, Proposition \ref{Prop tight W}
and Theorem \ref{Thm condition 02}, we can similarly prove the following result.

\begin{thm}\label{Thm condition 01}
$\Gamma^0(\int_0^{\cdot}\dot g(s)ds)$ is a continuous mapping from $g\in S_N$ into $C([0,T];\mathbb{V})$, in particular,
$\{\Gamma^0(\int_0^{\cdot}\dot g(s)ds);\ g\in S_N\}$ is a compact subset of $C([0,T],\mathbb{V})$.
\end{thm}


\def\refname{ References}

\end{document}